\documentclass[12pt,reqno,a4paper]{amsart}
\usepackage{amssymb,amsmath,amsthm,amscd}
\usepackage{mathtools}
\usepackage{esint}
\usepackage{mathrsfs}
 
\usepackage{mathabx}
\usepackage[T1]{fontenc}
\usepackage{hyperref}
\hypersetup{colorlinks,linkcolor={red},citecolor={blue}} 
\usepackage{enumerate}
\usepackage{tikz}
\usepackage{comment}

\newtheorem{theorem}{Theorem}
\newtheorem{lemma}[theorem]{Lemma}

\newtheorem{definition}[theorem]{Definition}

\newtheorem{proposition}[theorem]{Proposition}

\numberwithin{theorem}{section}
\numberwithin{equation}{section}

\def\N {{\mathbb N}}
\def\Z {{\mathbb Z}}

\def\R {{\mathbb R}}
\def\Rd {{\mathbb R}^{d}}
\def\C {{\mathbb C}}

\def\D {{\mathcal D}}

\DeclareMathOperator{\diff}{d\!}
\DeclareMathOperator{\supp}{supp}

\DeclareMathOperator{\dist}{dist}
\DeclareMathOperator{\sgn}{sgn}
\DeclareMathOperator{\diam}{diam}

\DeclareMathOperator{\pv}{p.v.}

\DeclareMathOperator{\Mod}{Mod}

\DeclareMathOperator{\SUM}{sum}
\DeclareMathOperator{\BDR}{bdr}
\DeclareMathOperator{\COR}{cor}
\DeclareMathOperator{\TOP}{top}
\DeclareMathOperator{\FULL}{full}
\DeclareMathOperator{\op}{op}
\DeclareMathOperator{\FT}{FT}
\DeclareMathOperator{\HL}{HL}

\DeclareMathOperator{\I}{I}
\DeclareMathOperator{\II}{II}
\DeclareMathOperator{\III}{III}
\DeclareMathOperator{\IV}{IV}
\DeclareMathOperator{\V}{V}

\def\<{\left\langle}
\def\>{\right\rangle}

\providecommand{\trm}[1]{\textrm{#1}}

\providecommand{\abs}[1]{ \left \lvert  #1 \right \rvert }

\providecommand{\no}[1]{  \lVert  #1  \rVert }

\providecommand{\nos}[1]{ \left \lVert  #1 \right \rVert }

\providecommand{\la}[1]{  \langle #1 \rangle}

\title[Uniform bounds]{Uniform bounds for 
bilinear symbols with linear K-quasiconformally 
embedded singularity}

\author{Marco Fraccaroli}
\author{Olli Saari}
\author{Christoph Thiele}

\address{Marco Fraccaroli, BCAM - Basque Center for Applied Mathematics, 
    Mazarredo, 14 E48009 Bilbao, 
    Basque Country, Spain.}
\email{mfraccaroli@bcamath.org}

\address{Olli Saari, Departament de Matem\`atiques, 
	Universitat Polit\`ecnica de Catalunya,
	Avinguda Diagonal 647, 08028 Barcelona,
	Catalunya, Spain}
\address{Centre de Recerca Matem\`atica, Edifici C, Campus Bellaterra, 08193 Bellaterra, Catalunya, Spain}
\email{olli.saari@upc.edu}

\address{Christoph Thiele, Mathematical Institute, 
	University of Bonn,
	Endenicher Allee 60, 53115 Bonn,
	Germany}
\email{thiele@math.uni-bonn.de}


\allowdisplaybreaks

\newcommand{\changelocaltocdepth}[1]{%
  \addtocontents{toc}{\protect\setcounter{tocdepth}{#1}}%
  \setcounter{tocdepth}{#1}%
}

\usepackage[normalem]{ulem}
\begin{document}
\begin{abstract}
We prove bounds in the strict local $L^{2}(\R^{d})$ range 
for trilinear Fourier multiplier forms with a $d$-dimensional singular subspace. 
Given a fixed parameter $K \ge 1$, 
we treat multipliers with non-degenerate singularity that are push-forwards by $K$-quasiconformal matrices 
of suitable symbols.
As particular applications,
our result recovers the uniform bounds for the one-dimensional bilinear Hilbert transforms in the strict local $L^{2}$ range,
and it implies the uniform bounds for two-dimensional bilinear Beurling transforms,
which are new,
in the same range.
\end{abstract}

\maketitle

\tableofcontents

\section{Introduction}

Let $d \ge 1$ and 
let $\Gamma_0$ be the linear
subspace of $\R^{3\times d}$
that consists of all vectors $(\xi_1,\xi_2,\xi_3)$ with
$\xi_1+\xi_2+\xi_3=0$.
Trilinear Fourier multiplier forms on $\Gamma_0$ 
are studied in order to understand mapping properties 
of bilinear Fourier multiplier operators on $\R^{d}$.
In the present paper,
we prove bounds in the strict local $L^{2}$ range for multipliers whose singular set can be written as an image of the $d$-dimensional diagonal of $\R^{3\times d}$ under a block $K$-quasiconformal matrix.
Our bounds depend on the matrix through the parameter $K$ alone, in this sense we prove bounds uniform
in isotropic dilations and rotations. 
We comment more on the motivation for such bounds after stating the main result.
 
We normalize the Fourier transform of a Schwartz function as 
\[\widehat{f}(\xi)=\int_{\R^d} f(x) e^{-2\pi i x \cdot \xi} \, \diff x 
.\]
Let $1 < p < \infty$. 
We denote the $L^{p}$ norm of a measurable function by 
\[ \no{f}_p^p \coloneq \int_{\R^d} \abs{ f(x) }^p \, \diff x.\]
Let $K \ge 1$.
A linear map
\[
L = L_1 \oplus L_2\oplus L_3 \]
mapping $\R^{3 \times d}$ to itself 
 is said to be block $K$-quasiconformal if 
for all $n \in \{1,2,3\}$ we have $L_n:\R^d\to \R^d$
and \[ 
\no{L_n}_{\op}^{d} \le K \det L_n.  
\]
 We say that $L$ is non-trivial if additionally 
\[
L_1  + L_2 + L_3 = 0.
\]

\begin{theorem}
\label{main-theorem}
Let $d \ge 1$, $K \ge 1$ and 
\[
\frac{1}{p_1}+\frac{1}{p_2}+\frac{1}{p_3} = 1, \quad 2 < p_1,p_2,p_3 < \infty.
\]
There exists a constant $C = C (d,K,p_1,p_2,p_3)$ such that the following holds.

Let $m:\R^{3 \times d} \to \C$ satisfy 
\begin{equation}
\label{e:tildesymbol}
\abs{ \partial_1^{\gamma_1}\partial_2^{\gamma_2}\partial_3^{\gamma_3} m(\xi) } \le \sup\{  \abs{ \xi - (\tau,\tau,\tau) }^{- \abs{ \gamma }} \colon \tau \in \Rd\}
\end{equation}
for all $\gamma \in \N^{3 \times d}$ with $\abs{ \gamma } \le 100d$.
Let $L$ be a non-trivial block $K$-quasiconformal matrix.
Define  
\[
\Lambda_m(f_1,f_2,f_3) = 
\int_{\R^{3 \times d}} \delta_0(\xi_1+\xi_2+\xi_3) \widehat{f}_1(\xi_1)\widehat{f}_2(\xi_2)\widehat{f}_3(\xi_3) m(L^{-1}\xi) \, \diff \xi
\]
where $\delta_0$ is the Dirac mass at origin.

Then for all triples of Schwartz functions $f_1$, $f_2$ and $f_3$ on $\R^{d}$
\[
\abs{ \Lambda_m(f_1,f_2,f_3) } \le C \prod_{n=1}^{3} \no{f_n}_{p_n}.
\] 
\end{theorem} 
We use a symbol $m$ defined on all of $\R^{3 \times d}$
for convenience, 
but instead of that, 
a symbol only defined on $\Gamma_0$ 
with conditions stated using directional
differential operators within the space $\Gamma_0$ could be used as well. 
Similarly,
the use of the mapping $L$ in the definition of the form is a compact way to express 
a set of certain anisotropic symbol estimates on $m$ through the simple condition \eqref{e:tildesymbol}.
We point out that the restriction of Theorem~\ref{main-theorem}
to the strict local $L^2$ range is likely not to be sharp.
Moreover, we do not see any obvious obstruction for an analogy of our result for higher orders of multilinearity. The only missing ingredient for the latter seems to be a suitable generalization of the uniform paraproduct estimate as in \cite{MR1945289}.
However,
we did not attempt any of these extensions in order to keep
the technicalities in this paper more limited and 
have better focus on some of the key ideas of our approach. 
For related work in $d=1$ extending the range of exponents the bilinear Hilbert transform, 
see \cite{MR1933076}, \cite{MR2320411}, \cite{gennady_uniform},
\cite{MR2997005}, and \cite{MR3453362}. 

The simplest interesting special case of the Theorem~\ref{main-theorem} is $d=K=1$, 
when $L=(L_1,L_2,L_3)$ is a vector of nonzero real numbers adding up to $0$ 
and 
\[m(\xi_1,\xi_2,\xi_3)=\frac{L_1\xi_1+L_2\xi_2+L_3\xi_3}
{\sqrt{(L_1\xi_1+L_2\xi_2+L_3\xi_3)^2+(\xi_1+\xi_2+\xi_3)^2}},\]
which restricted to the hyperplane $\xi_1+\xi_2+\xi_3=0$ reads as
\[m(\xi_1,\xi_2,\xi_3)=   \sgn (L_1\xi_1+L_2\xi_2+L_3\xi_3).\]
In this case, 
$\Lambda_m$ is a scalar multiple of the trilinear form dual to the bilinear Hilbert transform,
which can be written on the spatial side as
\begin{equation}
\label{bhtwithms}  
 \pv \iint_{\R^{2}}
f_1(x+M_1t)f_2(x+M_2t) f_3(x+M_3t)\, \frac{\diff x \diff t}{t} ,
\end{equation} 
where $M=(M_1,M_2,M_3)$ is a unit vector perpendicular to
both $(1,1,1)$ and $L$. 
No two components of $M$
are equal, 
because no component of $L$ is zero.
This condition is referred to as non-degeneracy of $M$.
The case of
\eqref{bhtwithms} with two components of the unit vector $M$ equal is called degenerate. 
If for example $M_3=M_1$, we have 
\[\Lambda_m(f_1,f_2,f_3)
= \int_\R f_1(x)  f_3(x)\left[ \pv \int_\R
f_2(x+t) \, \frac{\diff t}{t} \right]\, \diff x. \]
One obtains $L^p$ bounds for this form by H\"older's inequality 
and bounds for the linear Hilbert transform.
Bounds for the nondegenerate case of the bilinear Hilbert transform require a different argument
and were shown in the exponent range of Theorem~\ref{main-theorem} in \cite{MR1491450}, 
albeit with constants blowing up as $M$ tends to a degenerate value. 
Bounds uniform in $M$ were later proven in \cite{MR2113017} for the first time.
These results are covered by Theorem~\ref{main-theorem}.

The simplest example of our main theorem which is new is the case
where $d=2$, $K=1$ and $(L_1,L_2,L_3)$ is a triple of conformal matrices adding up to zero. 
In this case, we identify $\R^2$ with $\C$
and view the application of the  matrices $L_n$ as multiplication by complex numbers.
Moreover, we set
\[m(\zeta_1,\zeta_2,\zeta_3)=
\frac{(\overline{L_1\zeta_1+L_2\zeta_2+L_3\zeta_3})^2}
{ \abs{ L_1\zeta_1+L_2\zeta_2+L_3\zeta_3 }^2+ \abs{ \zeta_1+\zeta_2+\zeta_3 }^2.}
\]
Similar computations as for the bilinear Hilbert transform identify
$\Lambda_m$ as a scalar multiple of what one might call the bilinear Beurling transform
\[ \pv \iint_{\C^{2}} f_1(z+M_1\zeta)f_2(z+M_2\zeta) f_3(z+M_3\zeta) \, \frac{\diff A(z) \diff A(\zeta)}{\zeta^{2}} \]
where $A$ denotes the area measure.
Thus our main theorem implies $L^p$ bounds in the strictly locally $L^2$
range for the bilinear Beurling transform uniformly in $M$.
The Beurling kernel $\zeta^{-2}$ can be replaced by any standard Calder\'on--Zygmund kernel 
arising from a Mikhlin multiplier.

In dimension $d=1$, the cases for $L$ allowed in Theorem
\ref{main-theorem} together with a small number of easily understood degenerate cases provide an exhaustive picture of all cases of $L$.
The situation in higher dimensions is more complicated.
There are completely non-degenerate cases, 
completely degenerate cases in the sense that $L_n=0$ for some $n$, 
and further there is a zoo of distinct cases that one may call partially degenerate. 
For fixed $K$, 
our main theorem proves uniform bounds for the non-degenerate cases as 
one approaches the completely degenerate cases inside a cone 
that stays away from the partially degenerate cases. 
Within the conformal context,
our theorem covers all cases including the degenerate ones.
In this respect,
we show that the setting of one complex dimension is quite analogous 
to the setting of one real dimension.

Concerning the general case,
a list, not exhaustive, 
of five partially degenerate cases for $d=2$ were described in \cite{MR2597511},
and four of them were shown to be bounded,
albeit without any attempt to prove uniform bounds.
The remaining one,
called the twisted paraproduct,
was later treated in \cite{MR2990138} 
(see also \cite{MR3161332} for preliminary results 
and \cite{MR3488377,MR3683098} for further work).
A further partially degenerate case is the triangular Hilbert transform described in \cite{MR3482272}, where one dimension of the kernel is integrated out 
because it projects to zero in the arguments of all functions.
The triangular Hilbert transform is not known to satisfy any $L^p$ bounds, 
and it is well-understood that presently known techniques 
are insufficient to obtain such bounds.
A version of Theorem~\ref{main-theorem} with uniformity in $K$,
as opposed to our assumption on $K$ being fixed,
would imply bounds for the triangular Hilbert transform.
Bounds for the triangular Hilbert transform as well as some of the known bounds for other partially degenerate cases in $d=2$
would, in turn, 
imply bounds for the so-called Carleson operator in the corresponding $L^p$ spaces,
see \cite{MR199631}, \cite{MR340926} and \cite{MR0238019}. 
A more systematic classification of the partially degenerate cases appears in \cite{warchalski2019uniform}, 
where also some uniform bounds are proven in a discrete model.  

The main technical novelty of the current work is 
the application of our previous work \cite{fraccaroli2022phase},
where we improved and extended the method of phase plane projections, 
previously studied in \cite{MR1979774} in dimension one, 
to higher dimensions.
In order to apply the set-up introduced in \cite{fraccaroli2022phase},
we have to reformulate the standard phase space decomposition of the form $\Lambda_m$ in a new way.
Unlike the existing literature using either stopping times and outer measures, see \cite{MR3312633},
or a tree-selection algorithm with various size functionals acting on families of multitiles,
see \cite{MR1491450}, \cite{MR1933076} and \cite{MR2113017};
our proof arranges the tree-selection in a different way.
In particular,
unlike our main inspiration \cite{MR1979774},
we put emphasis on choosing the top intervals and top frequencies 
and let them define regions in phase space,
the trees.
Each tree, a region in the phase space, 
is then divided into a boundary and a core.
The treatise of these two parts can be separated into two independent modules.
The estimation of the boundaries is completely independent of paraproduct theory of any kind, 
just invoking H\"older's inequality.
The estimation of the cores in turn
relies on two real analysis lemmas, 
one on paraproduct estimates and one on phase space localization,
which are stand-alone results 
that do not make any explicit reference to the notion of a tree.
Clarifying the roles of the core part and the boundary part of a tree 
is the main insight we are communicating. 
Later,
at the level of tree selection,
we further notice that almost all non-trivial phase space interaction of the selected trees is encoded in their boundary parts. 
Summing up,
while the paraproduct theory of boundaries is very simple and that of cores more complicated,
the orders of complexity are swapped when carrying out the tree selection.
 
We close the introduction commenting a bit more on the background context of the study uniform bounds for multilinear operators. 
On one hand,
one may use uniform bounds over parametrized families of singular operators 
to conclude bounds for superpositions of these operators as the parameter varies. 
While integrable rather uniform dependence on the parameter may suffice for this purpose in some applications, 
even integrable dependence may need more work than the basic non-uniform bounds.
We refer to \cite{MR3231215}, \cite{MR3255002} and \cite{MR3293439} for a discussion about connections to Calder\'on commutators and the Cauchy integral over Lipschitz curves, 
as the original motivation for studying the bilinear Hilbert transforms. 
Secondly, multilinear forms whose multipliers are
characteristic functions of convex sets $E$ are closely related to uniform bounds for multipliers which are characteristic 
functions of half planes relative to tangent lines of $E$.
This connection appears in \cite{MR2701349}, \cite{MR2197068}, \cite{MR3000982}, \cite{MR2413217}, \cite{MR4566711},
and \cite{MR3337797}.

Finally,
we describe the structure of the present paper.
Section \ref*{sec:out} contains the outline of the proof of Theorem \ref{main-theorem},
which is organized into four propositions.
These principal propositions are proved in Sections \ref{s:globaltreeprop}, \ref{s:tree-estimate}, \ref{s:almostorthogonal} and \ref{s:whitneyprop}, one proposition in each section.
Theorem \ref{main-theorem} is deduced from the contents of the outline Section \ref{sec:out}
in Section \ref{s:main-theorem}.
Sections \ref{s:globaltreeprop}-\ref{s:whitneyprop}
are independent of each other 
and only make reference to Section \ref{sec:out}.
Section \ref{s:main-theorem} depends on arguments in Sections \ref{s:globaltreeprop}-\ref{s:whitneyprop} only through the propositions stated in Section \ref{sec:out}.
Section \ref{s:almostorthogonal} is slightly longer than its siblings,
and it is divided further into an outline part and 
five further numbered subsections,
which only refer to Section \ref{sec:out} and the overview part of Section \ref{s:almostorthogonal}. 
The sections that are a part of the logical scheme just described appear in the table of contents
while subtitles beyond the logical scheme do not appear there. 

\bigskip

\noindent
\textit{Acknowledgement.}
The authors were funded by the Deutsche Forschungsgemeinschaft (DFG, German Research Foundation) 
under the project numbers 390685813 (EXC 2047: Hausdorff Center for Mathematics) and 211504053 (CRC 1060: Mathematics of Emergent Effects).
The first author was supported by the Basque Government through the BERC 2022-2025 program and by the Ministry of Science and Innovation: BCAM Severo Ochoa accreditation CEX2021-001142-S / MICIN / AEI / 10.13039/501100011033.
The second author was supported by Generalitat de Catalunya (2021 SGR 00087), 
Ministerio de Ciencia e Innovaci\'on and the European Union -- Next Generation EU (RYC2021-032950-I),  (PID2021-123903NB-I00) and the Spanish State Research Agency 
through the Severo Ochoa and María de Maeztu Program for Centers and Units of Excellence in 
R\&D (CEX2020-001084-M).

\section{Outline of the proof} 
\label{sec:out}
We fix the dimension $d\ge 1$, 
dilation parameters $k_2 > k_1 > k_0 \ge 3$
with $k_i - k_j > 100d$ for $0\le j< i \le 2$
and the triple of exponents $(p_1,p_2,p_3)$ 
satisfying 
\[
\frac{1}{p_1}+\frac{1}{p_2}+\frac{1}{p_3} = 1, \quad 2 < p_1,p_2,p_3 < \infty.
\]
Let
\begin{equation*}\label{defineepsilon}
\varepsilon= \min \{ p_1 - 2, p_2-2, p_3 - 2 \}.
\end{equation*}
In addition,
we fix a number $\alpha > 2d$, $\alpha < 8 d$.
We further fix linear maps $L_1$, $L_2$ and $L_3$ as in Theorem~\ref{main-theorem}. 
%We define $L : \R^{3 \times d} \to \R^{3 \times d}$ by setting $Lz = (L_1z_1,L_2z_2,L_3z_3)$.
For $n\in \{1,2,3\}$, 
we choose $v_n \in \Z$ such that 
\[
2^{v_n-1} <  \no{L_n}_{\op} \le 2^{v_n}.
\] 
Fix an index $n_*\in \{1,2,3\}$ such that
\begin{equation}
v_{n_*}=\min \{ v_1, v_2, v_3 \}.
\end{equation}
As the condition \eqref{e:tildesymbol} is invariant under scaling $\xi \mapsto \lambda\xi$, 
we may assume that $v_{n_*}=0$.

Denote by $B(x,r)$ the open ball centered at $x \in \Rd$ and with radius $r$.
For $\xi \in \R^{3\times d}$, $r>0$, and $n \in \{1,2,3\}$, 
define $Q_n(\xi,r) \subset \R^{d}$ to be the minimal open rectangular box with sides parallel to the coordinate axes
containing $B(\xi_n, 2^{v_n}r)$.
Let 
\[
Q(\xi,r) = Q_1(\xi,r)  \times Q_2(\xi,r)  \times Q_3(\xi,r) .
\]
Let
\[
\Gamma = \{ (L_1\tau,L_2\tau,L_3\tau) \colon \tau \in \R^{d} \}.
\]

Let $\mathcal{W}$ be a maximal set of pairwise disjoint 
rectangles of the form $Q(\xi,2^{-j})$ with  $\xi \in \R^{3\times d}$ and $j\in \Z$  with
\begin{equation*}\label{whitney1}
\Gamma \cap  Q(\xi, 2^{k_0-j})=\emptyset
\end{equation*}
and \begin{equation*}\label{whitney2}
\Gamma\cap Q(\xi, 2^{k_0+1-j}) \neq \emptyset.  
\end{equation*} 
%Let $\mathcal{W}$ be the set of all rectangles $Q(\xi,2^{-j})$ with $(\xi,j) \in \mathcal{W}_0$.
% We refer to the elements of $\mathcal{W}$ as frequency boxes.
For all $N > 0$ let $\mathcal{W}_N$ be the finite subset of $\mathcal{W}$ defined by
\begin{equation*}
    \mathcal{W}_N = \{ Q(\xi,2^{-j}) \in \mathcal{W} \colon \abs{\xi}, \abs{j} \leq N \} .
\end{equation*}

For a cube with sides parallel to the coordinate axes $I \subset \R^d$, define  the mollified distance $\rho_I$ by
\begin{equation*}\label{definerhoi}
\rho_I(x)=\inf\{r>1 \colon x\in (2r-1)I\},
\end{equation*}
where $aI$ denotes the cube with same center as $I$ and $a$ times the side-length. Moreover, for a Borel set $F \subset \R^d$, define
\begin{equation*}\label{definerhoi_2}
\rho_I(F)=\inf\{ \rho_I(x) \colon x \in F \}.
\end{equation*}

\begin{definition}[Frequency cut-offs]
\label{def:freqsupp}

Let $E \subset \R^{3 \times d}$ be bounded with open interior.  Define $\Phi_{n}^{\alpha}(E)$ to be the set of continuous complex valued functions $\phi$ on $\Rd$ with 
\[\abs{ \phi(x) }\le  2^{(v_n-j)d} \rho_{[0,2^{j-v_n})^{d}}^{-\alpha}(x)  \]
for all $x\in \Rd$ and 
\[\supp \widehat{\phi} \subset \{\xi_n:\xi \in E \},\]
where $j\in \Z$ is maximal such that
there exists $\xi \in \R^{3 \times d}$ with $E \subset Q(\xi,2^{-j})$.

\end{definition}

In Section~\ref{s:main-theorem},
Theorem~\ref{main-theorem} 
is reduced to Proposition~\ref{whitneyprop} below,
where the multiplier is replaced by a sum of tensor multipliers.

\begin{proposition}[Weak estimate for tensor model]
\label{whitneyprop}
Let 
\begin{equation*}
\frac{1}{q_1}+\frac{1}{q_2}+\frac{1}{q_3} = 1, \quad 2 < q_1,q_2,q_3 < \infty.
\end{equation*}
There exists a constant $C = C(d,\alpha,k_0,q_1,q_2,q_3)$ such that the following holds.

For $Q \in {\mathcal{W}}$ and $n\in \{1,2,3\}$, 
let $\phi_{Q,n} \in \Phi_{n}^{4\alpha}(Q)$. 
%Let $E_1$, $E_2$ and $E_3$ be measurable sets in $\R^d$ with finite measure and 
For each $n\in \{1,2,3\}$ let $f_n \in L^2(\R^d)$ be a function
%measurable function with
such that
\begin{equation*}\label{l2normalized}
\no{f_n}_{\infty} \le 2. %, \quad \no{f_n}_{2} \le 2 |E_{n}|^{1/2}.
\end{equation*}
Then, for all $N>0$,
\begin{equation}\label{e:tensorest}
% \abs{ \sum_{\substack{Q(\xi, 2^{-j}) \in {\mathcal{W}}\\ \abs{ \xi }, \abs{ j } \le N}}
\abs{ \sum_{ Q \in \mathcal{W}_N}
  \int_{\R^d} \prod_{n=1}^3 [ \phi_{Q,n} * {f}_n (x)] \, \diff x }
  \le C \prod_{n=1}^3 \no{f_n}^{2/q_n}_{2} %|E_1|^{1/q_1}|E_2|^{1/q_2}|E_3|^{1/q_3} 
  .
\end{equation}

\end{proposition}

The proof of Proposition~\ref{whitneyprop} can be found in Section~\ref{s:whitneyprop}.
It requires several intermediate results,
which we state next.
The following frequency localized version of Proposition~\ref{whitneyprop} 
will play a role inside the proof.
While the singularity of the bilinear multiplier in Proposition~\ref{whitneyprop} 
can still be truly $d$-dimensional,
Proposition~\ref{globaltreeprop} only deals with a point singularity in the spirit of more classical Coifman-Meyer multilinear multipliers.
Proposition~\ref{globaltreeprop} will be proven in Section~\ref{s:globaltreeprop}. 

\begin{proposition}[Frequency localized estimate]
\label{globaltreeprop}
Let $k$ be a positive integer and
\begin{equation*}
\frac{1}{q_1}+\frac{1}{q_2}+\frac{1}{q_3} = 1, \quad 2 < q_1,q_2,q_3 < \infty.
\end{equation*}
There exists a constant $C = C(d,\alpha,k_0,k,q_1,q_2,q_3)$ such that the following holds.

Let $\eta\in \Gamma$.
%Let ${\mathcal{V}}$ be a finite subset of ${\mathcal{W}}$
%such that the sets $ Q \in \mathcal{V}$ are pairwise disjoint
%and in addition that for every $Q \in \mathcal{V}$ we have 
%\begin{equation}\label{tip}
%\eta \in 2^{k}Q.
%\end{equation}
For $Q\in \mathcal{W}$ and $n\in \{1,2,3\}$, 
let $\phi_{Q,n} \in \Phi_{n}^{4\alpha}(Q)$. 
For each $n\in \{1,2,3\}$ let $f_n \in L^{q_n}(\R^d)$. Then, for all $N > 0$,
\begin{equation*}\label{e:localest}
% \abs{ \sum_{\substack{Q(\xi, 2^{-j}) \in {\mathcal{W}}\\ \abs{j} \le N,\ \eta \in 2^{k}Q }} 
\abs{ \sum_{\substack{Q \in \mathcal{W}_N \\ \eta \in 2^{k}Q }} 
  \int_{\R^d} \prod_{n=1}^3 [ \phi_{Q,n} * {f}_n (x) ] \, \diff x }
  \le C \prod_{n=1}^3 \no{f_n}_{q_n} .
\end{equation*}  
\end{proposition}
The reduction of Proposition~\ref{whitneyprop} to Proposition~\ref{globaltreeprop} 
features a stopping time argument, 
which introduces spatial truncations in addition to the mere frequency localization discussed so far
and utilizes the notion of trees defined below.

For $k \in \Z$, let $\mathcal{D}_k = \{2^{k}([0,1)^{d} + l)  \colon l \in \Z^{d}\}$ 
and $\mathcal{D} = \bigcup_{k \in \Z} \mathcal{D}_k$.
An element of $\mathcal{D}$ is called a dyadic cube.

\begin{definition}[Multitile, $n$-tile]

A product $I \times Q$ is called a multitile if $I \in \D$ and $Q\in \mathcal{W}$  and $ \abs{ Q_{n_{*}} } ^{-1} = \abs{I} $.
For a multitile $I \times Q$ and  $n \in \{1,2,3\}$,
we call the product $I \times Q_n$ an $n$-tile. 
If $P = I \times {Q}$ is a  multitile,
we write $I_P$ for  $I$ and
 $Q_P$ for  ${Q}$. 
\end{definition}

\begin{definition}[Tree]
\label{def:tree}
Let $\mathcal{V}$ be a finite subset of multitiles,
let $\xi \in \Gamma$, and let $I_0 \in \mathcal{D}$.
Assume there exists at least one $P \in \mathcal{V} $ with $I_P = I_0$ and 
$\xi \in 2^{k_2+1} Q_P $.
Then the triple $(\xi,I_0,\mathcal{V})$ defines a tree $T$. We write $\xi_T $ for $\xi$, $I_T$ for $I_0$, $\mathcal{V}_T$ for $\mathcal{V}$, and $j_T$ for the top scale $\log_{2} \abs{I_0}^{1/d}$.
Attached to the tree $T$ are the following objects.
\begin{itemize}
  \item The family $\mathcal{P}_T$ of multitiles 
  % with $Q_P \in \mathcal{V}$ and
  in $\mathcal{V}$ with
   $I_P \subset  I_T$ and \[\xi_T \in 2^{k_2+1} Q_P.\]
  \item The family $\mathcal{B}_T$ of multitiles 
  $P \in \mathcal{P}_T$ with 
  \[\xi_T \in 2^{k_2+1} Q_P \setminus 2^{k_1+1} Q_P .\]
  \item The family  $\mathcal{I}_T$ of dyadic cubes $I\in \mathcal{D}$ such that there exist $ P$ and $P'$ in $ \mathcal{P}_T \setminus \mathcal{B}_T$ with 
  $I_P \subset I \subset I_{P'}$.
\end{itemize}  
\end{definition}

The following definition gives a gauge to 
the size of a functions near a tree.
%For a dyadic cube $I$, define  the mollified distance $\rho_I$ by
%\begin{equation*}\label{definerhoi}
%\rho_I(y)=\inf\{r>1: y\in (2r-1)I\},
%\end{equation*}
%where $aI$ denotes the cube with same center as $I$ and $a$ times the side-length.
\begin{definition}[Main sizes]
\label{def:main-sizes}
Let $1 \le p \le \infty$,
$n \in \{1,2,3\}$ and 
$f \in L^{p}(\R^{d})$.
Let $T$ be a tree.
We define
\begin{align*}
\Sigma_{n,p,f}^{\BDR}(T)  &= 
\sup_{P \in \mathcal{B}_T}
\sup_{\phi \in \Phi_{n}^{4\alpha}(Q_P)}    
\frac{\no{\rho_{I_P}^{-\alpha} [ \phi* f ] }_p}{ \abs{I_P}^{1/p}}, \\
\Sigma_{n,f}^{\SUM}(T)  &= 
 \left(  \frac{1}{\abs{I_T}} \sum_{P \in \mathcal{B}_T} 
\sup_{\phi \in \Phi_{n}^{4\alpha}(Q_P)}  
 \no{1_{I_P} [ \phi* f ] }_2^{2} \right)^{1/2}, \\
\Sigma_{n,p,f}^{\COR}(T)  &=   \sup_{i \in \Z}
\sup_{I \in \D_i \cap \mathcal{I}_T}\  
\sup_{\phi \in \Phi_{n}^{4\alpha}(Q_{T,i})}  
\frac{\no{\rho_I^{-\alpha} [ \phi* f ]}_p}{\abs{I}^{1/p}}, 
\end{align*} 
where $Q_{T,i} = Q(\xi_T,2^{k_1+5d-i})$ and $1/\infty$ is understood to be $0$.
\end{definition}
Heuristically,
the core size is large enough to control a phase space paraproduct,
but it is slightly too imprecise in terms of phase space localization.
In order to maintain the information about frequency localization of a tree,
the frequencies seen as peripheral with respect to the top-frequency 
must be measured with a different kind of size, the sum size.
The pair of sum-size and core-size are together strong enough to control the paraproduct and maintain the phase space localization,
but in order to sum together the trees of different amplitudes,
this couple still fails by a logarithmic blowup.
To adjust this last piece,
a multiplicative fraction of the sum-size is replaced by the boundary size,
which is a sup-size again,
but of nature lacunary with respect to the top frequency.
After this last adjustment,
the triple of sizes succeeds in the task of controlling the paraproduct,
maintaining phase space localization
and recovering summability over amplitudes.  
In the following proposition,
we control the phase space paraproduct by the sizes.
The proof can be found in Section~\ref{s:tree-estimate}.
 
\begin{proposition}[Phase space localized estimate]
\label{thm:tree-estimate}
Let 
\begin{equation*}
\frac{1}{q_1}+\frac{1}{q_2}+\frac{1}{q_3} = 1, \quad 2 < q_1,q_2,q_3 < \infty.
\end{equation*}
There exists a constant $C = C(d,\alpha,k_0,k,q_1,q_2,q_3)$ such that the following holds.

Let $T $ be a tree.
For each $P \in \mathcal{P}_T$ and $n \in \{1,2,3\}$,
let $\phi_{P,n} \in \Phi_{n}^{4\alpha}(Q_P)$.
Then for any $n' \in \{1,2,3\}$
\begin{multline}
\label{eq:single-tree-lac}
\abs{ \sum_{ P \in \mathcal{B}_T } \int_{\R^d} 1_{I_P} (x) \prod_{n=1}^{3} [ \phi_{P,n} * f_n(x) ] \, \diff x } \\ \le C \abs{I_T} \Sigma_{n',\infty,f_{n'}}^{\BDR}(T)   \prod_{n \ne n'} \Sigma_{n,f_n}^{\SUM}(T)   ,
\end{multline}
\begin{equation} 
\label{eq:single-tree-cor}
\abs{ \sum_{ P \in \mathcal{P}_T \setminus \mathcal{B}_T } \int_{\R^d} 1_{I_P} (x) \prod_{n=1}^{3} [ \phi_{P,n} * f_n(x) ] \, \diff x } 
 \le C \abs{I_T} \prod_{n = 1}^{3} \Sigma_{n,q_n,f}^{\COR}(T) .    
\end{equation}
% \begin{align}
% \label{eq:single-tree-lac}
% \abs{ \sum_{ P \in \mathcal{B}_T } \int_{\R^d} 1_{I_P} (x) \prod_{n=1}^{3} [ \phi_{P,n} * f_n(x) ] \, \diff x } & \le C \abs{I_T} \Sigma_{n',\infty,f_{n'}}^{\BDR}(T)   \prod_{n \ne n'} \Sigma_{n,f_n}^{\SUM}(T)   , \\
% \label{eq:single-tree-cor}
% \abs{ \sum_{ P \in \mathcal{P}_T \setminus \mathcal{B}_T } \int_{\R^d} 1_{I_P} (x) \prod_{n=1}^{3} [  \phi_{P,n} * f_n(x) ] \, \diff x } 
%  &\le C \abs{I_T} \prod_{n = 1}^{3} \Sigma_{n,p_n,f}^{\COR}(T) .
% \end{align}  
\end{proposition}

The remaining ingredient of the proof of Proposition~\ref{whitneyprop} 
is a partition of the set of all multitiles into trees,
to which Proposition~\ref{thm:tree-estimate} can be applied.
This last proposition will be proved in Section~\ref{s:almostorthogonal}.
% {\bf  This is where distinguishing the size applied to core tiles and boundary tiles becomes crucial. CT: Where else would it be? The beauty of being short is that things become self-explanatory and one can erase further. }

% \begin{definition}[Unhollow trees]
% Let $T$ be a tree. We set 
% \[
% \Theta(T) = \begin{cases}
%   1, \quad \trm{if there exists $P \in \mathcal{P}_T$ with $2^{k_1+1}Q_P \ni \xi_{T}$}. \\
%   0, \quad \trm{otherwise}.
% \end{cases}
% \]  
% \end{definition}

\begin{proposition}[Decomposition of the phase space]
\label{almostorthogonal}
There exists a constant $C = C(d,\alpha,k_0,k_1,k_2)$ such that the following holds.

%Let $\mathcal{W}$ be as in Proposition \ref{whitneyprop}. Then for each $M \in \Z \cup \{-\infty\}$ there exists a family of trees $\mathcal{T}_M$  
Let $N, N' > 0$. Let $\mathcal{V}$ be the finite subset of multitiles defined by
% \[
% \mathcal{V} = \{ P \colon Q_P \in \mathcal{W}_N , \ I_P = 2^k([0,1)^d + l), \abs{l} \leq N' \}
% \]
\[
\mathcal{V} = \{ P \colon Q_P \in \mathcal{W}_N , \ I_P \subset [ - N' 2^N, N' 2^N ]^{3 \times d} \}
\]
with $\mathcal{W}$ as in Proposition~\ref{whitneyprop}.
For each $M \in \Z \cup \{-\infty\}$ there exists a family of trees $\mathcal{T}_M$  
such that
\begin{equation*}
\mathcal{V} = \bigcup_{M \in \Z \cup \{-\infty \} } \bigcup_{T \in \mathcal{T}_M} \mathcal{P}_T  
\end{equation*}
% exhausts all the multitiles 
and the following hold for each $n \in \{1,2,3\}$.
\begin{itemize}
  \item For each tree $T \in \mathcal{T}_M$ for which there exists $P \in \mathcal{P}_T$ with $2^{k_1+1}Q_P \ni \xi_{T}$, it holds 
\[
 \Sigma_{n,2,f_n}^{\COR}(T) \le 2^{M/2}  \no{f_n}_{2}.
\]
  \item For every tree $T \in \mathcal{T}_M$, it holds  
\[
  \Sigma_{n,2,f_n}^{\BDR}(T) + \Sigma_{n,f_n}^{\SUM}(T) \le 2^{M/2}  \no{f_n}_{2}.
\]
  \item For every tree $T$ % on $\mathcal{W}$
 with $\mathcal{V}_T \subset \mathcal{V}$, it holds 
\begin{equation}
\label{e:size-by-linfty}
   \Sigma_{n,2,f_n}^{\COR}(T) + \Sigma_{n,2,f_n}^{\BDR}(T) + \Sigma_{n,f_n}^{\SUM}(T) \le C \no{f_n}_{\infty} . 
\end{equation}
  \item It holds 
\begin{equation}
\label{e:almorth-overlap}
\sum_{T \in \mathcal{T}_M} 2^{M} \abs{I_T} \le C .
\end{equation}
\end{itemize}
\end{proposition}
  
\changelocaltocdepth{1}

\subsection*{Complementary notation}
We conclude the section on outline of the proof by
listing some notational conventions that 
we intentionally omitted when describing the strategy of the proof 
but which will be worth memorizing for reading the proofs.
In what follows,
constants $C$ will depend on $d$, $\alpha$, $\varepsilon$, $k_0$, $k_1$, and $k_2$.  
The exact dependence will be implicit in our arguments.
We occasionally use the shorthand notation $A \lesssim B$ 
when $A \le CB$ for such a constant $C$.

% Concerning a tree based on $(\xi,I_0,\mathcal{V})$,
% see Definition \ref{def:tree},
% we write $j_T := \log_{2} \abs{I_0}^{1/d}$ for the top scale
% and $\mathcal{V}_T := \mathcal{V}$ for the reference family. 
% The word frequency box refers to a rectangle $Q \in \mathcal{W}$.
% A frequency box $Q \in \mathcal{W}$ is said to be subordinate to a frequency $\xi \in \Gamma_0$ if $\xi \in 2^{k_2+1} Q $.
% A multitile $I \times Q$ is said to be subordinate to a frequency $\xi$ if $Q$ is subordinate to $\xi$. 

Concerning the frequency cut-offs,
see Definition \ref{def:freqsupp},
we use the following shorthand notations:
\begin{itemize} 
  \item Given $\xi \in \R^{3 \times d}$ and $j \in \Z$, we denote 
\[
  \Phi_{n,j}^{\alpha}(\xi) = \Phi_{n}^{\alpha}(Q(\xi,2^{-j})).
\]
  \item Given $\xi \in \R^{3 \times d}$ and $j \in \Z$, we denote
\[
  \Psi_{n,j}^{\alpha}(\xi) = \Phi_{n}^{\alpha}(Q(\xi,2^{-j}) \setminus Q(\xi,2^{-j-2})).
\]
  \item Given $\xi \in \R^{3 \times d}$, we denote by $M_n(\xi,E)$ the set of $\phi$ such that 
\[
\sup_{\tau \in \R^{d} \setminus \{\xi_n\}} \abs{ (\tau-\xi_n)^{\beta}\partial^{\beta} \widehat{\phi}(\tau) } \le 2^{-v_n \abs{\beta} }, \quad \supp \widehat{\phi} \subset E 
\]
for all $\beta \in \N^{d}$ with $\abs{ \beta } \le 100d$.
We call such a $\phi$ a normalized $n$-Mikhlin cut-off to $E$ at $\xi$.
\end{itemize}

\section{Proof of Proposition~\ref{globaltreeprop}. Paraproduct}
\label{s:globaltreeprop}

Let $\eta$ and $\phi_{Q,n}$ be given as in Proposition~\ref{globaltreeprop}.
By a translation on the Fourier transform side we may assume $\eta=0$.
By definition of $\mathcal{W}$, for each $ Q \in \mathcal{W}$ it holds $0 \notin 2Q$.
Hence there exists $n\in \{1,2,3\}$ such that $0 \notin 2Q_n$.
By splitting into three cases and estimating \eqref{e:tensorest} in each case separately,
we may assume without loss of generality that $0 \notin 2Q_1$ for all $Q\in \mathcal{W}$. Further,
for each $j \in \Z$
there exists at most $C(d,k)$
distinct elements $Q \in \mathcal{W}$
with $2^{k}Q \ni 0$ and $ \abs{ Q_{n_{*}} } = 2^{-jd}$.
By splitting into $C(d,k)$ further subcases,
we may assume % for each $j \in \Z$ there exists at most one $Q \in \mathcal{W}$
% with $ \abs{ Q_{n_{*}} } = 2^{-jd}$
there exists at most one such $Q$.
Even further,
for each $Q$ with $ \abs{ Q_{n_{*}} } = 2^{-jd}$,
there exist $C(d,k)$, $\widetilde{C}(d,k)$, and $\{ c_{j',n} \colon \abs{ c_{j',n} } \leq \widetilde{C}(d,k) , \abs{ j' - j } \leq C(d,k) \}$
such that
\[
\phi_{Q,n} = \sum_{j' \colon \abs{ j' - j } \le C(d,k) } c_{j',n} \phi_{j',n} , \quad \phi_{j',n} \in \Phi_{n,j'}^{4\alpha}(0).
\]
Hence we may further reduce the study to the case where $\phi_{Q,n}$
is replaced by $\phi_{j,n}$ as above
and $v_n$ replaced by $v_{n}'$ with $\abs{ v_n - v_n' } \le C(d,k)$.
Hence we aim at bounding 
\[
\abs{ \sum_{j \in \mathcal{N}} 
  \int_{\R^d} c_j \prod_{n=1}^3 [ \phi_{j,n} * {f}_n (x) ] \, \diff x }
\]
where $\mathcal{N} \subset \Z$ is finite, 
$\phi_{j,1}\in \Psi_{1,j}^{4\alpha}(0)$ 
and $\phi_{j,n}\in \Phi_{n,j}^{4\alpha}(0)$ for $n \in \{2,3\}$.

Let $\chi$ be a Schwartz function on $\R^d$ such that $\widehat{\chi}(\tau)=0$ for $\abs{ \tau } \ge 2$ and $\widehat{\chi}(\tau)=1$ for $\abs{ \tau } \le 1$.
Define for $l \in \Z$
\[\chi_l(x)=2^{-l d}\chi(2^{-l}x)\]
and for each $j \in \mathcal{N}$ and $n\in \{2,3\}$, 
define
\[
\widehat{\rho}_{j,n}=\widehat{\phi}_{j,n}-\widehat{\phi}_{j,n}(0) \widehat{\chi_{j-v_n}} .
\]
By the triangle inequality, 
it suffices to prove for any collection
\[
\{c_{j} \colon \abs{ c_{j} } \le 1, \
j \in  \mathcal{N} \}
\]
bounds for the tree expressions 
\begin{align}
\label{globala}
\I &=  \abs{\sum_{j \in \mathcal{N}} 
  \int_{\R^d} \left[
  c_{j} \phi_{j,1} * {f}_1 (x) \right]
  \left[
  \rho_{j,2} * {f}_2 (x) \right]
  \left[
  \phi_{j,3} * {f}_3 (x) \right]
  \, \diff x}  ,\\
\label{globalb}
\II &=  \abs{\sum_{j \in \mathcal{N}} 
  \int_{\R^d} \left[
  c_{j} \phi_{j,1} * {f}_1 (x) \right]
  \left[
  \chi_{j-v_2} * {f}_2 (x) \right]
  \left[
  \rho_{j,3} * {f}_3 (x) \right]
  \, \diff x} ,
\\
\label{globalc}
\III &=  \abs{\sum_{j \in \mathcal{N}} 
  \int_{\R^d} \left[
  c_{j} \phi_{j,1} * {f}_1 (x) \right]
  \left[
  \chi_{j-v_2} * {f}_2 (x) \right]
  \left[
  \chi_{j-v_3} * {f}_3 (x) \right]
  \, \diff x }
\end{align}
separately.

We begin with \eqref{globala}. We estimate it with Cauchy-Schwartz in $\mathcal{N}$ 
and H\"older in $\R^d$ by
\[
     \nos{ \Big( \sum_{j \in \mathcal{N}}
\abs{ \phi_{j,1} * {f}_1 }^2 \Big)^{1/2}}_{q_1}
    \nos{ \Big( \sum_{j \in \mathcal{N}}
\abs{ \rho_{j,2} * {f}_2 }^2 \Big)^{1/2}}_{q_2}
     \nos{\sup_{j \in \mathcal{N}}
\abs{ \phi_{j,3} * {f}_3 } }_{q_3}.
\]
The term \eqref{globalb} is estimated similarly by
\begin{equation*}
\nos{ \Big(\sum_{j \in \mathcal{N}}
\abs{ \phi_{j,1} * {f}_1 }^2 \Big)^{1/2}}_{q_1}
\nos{\sup_{j \in \mathcal{N}}
\abs{ \chi_{j-v_2} * {f}_2 }}_{q_2}
\nos{ \Big(\sum_{j \in \mathcal{N}}
\abs{ \rho_{j,3} * {f}_3 }^2 \Big)^{
1/2}}_{q_3}.
\end{equation*}
In both cases,
we can apply the standard square function estimate (see Theorem~5.1.2 in \cite{MR2445437}) and maximal function estimates to 
obtain the desired bound. 
This completes the proof for $\I$ and $\II$.

It remains to estimate \eqref{globalc}. 
We telescope $\chi_{j-v_2}$ and $\chi_{j-v_3}$ into functions 
$\psi_l \coloneq \chi_{l-1}-\chi_{l}$
and thus write 

\begin{equation} \label{globalm} 
\begin{multlined}
\III \lesssim \sum_{m_1 = -v_1 - 10}^{-v_1+10} \\ 
\times \abs{ \sum_{ \substack{ m_2 \ge -v_2 \\ m_3 \ge -v_3} } \sum_{j \in \mathcal{N}} 
  \int_{\R^d}  [ \phi_{m_1+j} * {f}_1 (x) ] \prod_{n=2}^3 [ \psi_{m_
  n+j} * {f}_n (x) ] \, \diff x }
  % [ \psi_{m_2+j} * {f}_2 (x) ] [ \psi_{m_3+j} * {f}_3 (x) ] \, \diff x } 
\end{multlined}
\end{equation}
where $\phi_{m_1+j} =  \phi_{j,1} * \psi_{m_1+j}$.

Fix a triple $(\kappa_1,\kappa_2,\kappa_3) \in \Z^{3}$ and 
restrict the sums to $m_n \in \kappa_n + 1000d\Z$ for $n \in \{2,3\}$
and $j \in \widetilde{\mathcal{N}} = \kappa_1 +1000d\Z$. 
By triangle inequality and summation over the $(1000d)^{3}$ values of $(\kappa_1,\kappa_2,\kappa_3)$,
it suffices to bound the restricted sum.
Consider then a fixed term in the sum \eqref{globalm}.
Such a term is non-zero only if 
\[
0 \in (\supp \widehat{\phi}_{m_1+j} + \supp \widehat{\psi}_{m_2+j} + \supp \widehat{\psi}_{m_3+j}  ).
\]
Recalling that we work with indices modulo $1000d$,
this happens only if two of the numbers in $\{m_1,m_2,m_3\}$ are equal 
and the remaining one is larger.

Assume first $m_1 = m_n \le m_{n'}$  for fixed $n, n' \in \{2,3\}$.
Then, for $\mathfrak{m} = \max(m_1,-v_{n'})$, we bound \eqref{globalm} by
\begin{multline*}
\abs{  \sum_{m_{n'} \ge \mathfrak{m}} \sum_{j \in  \widetilde{\mathcal{N}}} 
  \int_{\R^d}  [ \phi_{m_1+j} * {f}_1 (x) ] [ \psi_{m_1+j} * {f}_n (x) ] [ \psi_{m_{n'}+j} * {f}_{n'} (x) ] \, \diff x }  \\
\le \int_{\R^d} \sum_{j \in  \widetilde{\mathcal{N}}}  \abs{ \phi_{m_1+j} * {f}_1 (x) } \abs{ \psi_{m_1+j} * {f}_{n} (x) } \abs{ \sum_{m_{n'} \ge \mathfrak{m}} \psi_{m_{n'}+j} * {f}_{n'} (x) } \, \diff x \\
\le \nos{ \Big(\sum_{j \in  \widetilde{\mathcal{N}}} \abs{ \phi_{m_1+j} * {f}_1 }^{2} \Big)^{1/2}  }_{q_1} \nos{ \Big( \sum_{j \in  \widetilde{\mathcal{N}}} \abs{ \psi_{m_1+j} * {f}_n }^{2} \Big)^{1/2}  }_{q_n} \\
\times
\nos{ \sup_{j \in  \widetilde{\mathcal{N}}} \abs{ \sum_{m_{n'} \ge \mathfrak{m} + j} \psi_{m_{n'}} * {f}_{n'} }  }_{q_{n'}}.   
\end{multline*}
These factors are bounded by the square function estimate and maximally truncated singular integral estimate,
which completes the proof in this case.

Assume then that $m_2 = m_3 \leq m_1$.
Now, for $\mathfrak{m} = \max(-v_{2},-v_{3})$, we bound \eqref{globalm} by
\begin{multline*}
\abs{   \sum_{j \in  \widetilde{\mathcal{N}} } 
  \int_{\R^d} [ \phi_{m_1+j} * {f}_1 (x) ]  \sum_{k = j+ \mathfrak{m}}^{j + m_1} [ \psi_{k} * {f}_2 (x) ] [ \psi_{k} * {f}_{3} (x) ] \, \diff x } \\
= \abs{   
  \int_{\R^d}    \sum_{k \in \Z} [ \psi_{k} * {f}_{2} (x) ] [ \psi_{k} * {f}_{3} (x) ] \sum_{ j \in \widetilde{\mathcal{N}} \cap \{k - m_1, \dots, k - \mathfrak{m} \} } [ \phi_{m_1+j} * {f}_1 (x) ] }  \, \diff x \\
\le 
\nos{ \Big( \sum_{k \in \Z} \abs{ \psi_{k} * {f}_2 }^{2})^{1/2} }_{q_2}
\nos{ \Big( \sum_{k \in \Z} \abs{ \psi_{k} * {f}_3 }^{2} \Big)^{1/2} }_{q_3} \\
\times
\nos{ \sup_{k \in \Z} \abs{ \sum_{ j \in \widetilde{\mathcal{N}} \cap \{k - m_1, \dots, k - \mathfrak{m} \} } \phi_{m_1+j} * {f}_1 } }_{q_1}.
\end{multline*}
Again,
the bound follows by the square function estimate and maximally truncated singular integral estimate
and the proof is complete. \qed

\section{Proof of Proposition~\ref{thm:tree-estimate}. Tree estimate}
\label{s:tree-estimate}
 
\subsection*{Boundary part}
Given any family of multitiles $\mathcal{F} \subset \mathcal{P}_T$,
we denote 
\begin{align*}
\Lambda_{\mathcal{F}}(f_1,f_2,f_3) &= \sum_{P \in \mathcal{F}} \int_{\R^d} 1_{I_P}(x) \prod_{n=1}^{3} [ \phi_{P,n}*f_n(x) ] \, \diff x.
\end{align*}
% We aim at proving a bound on $\Lambda_{\mathcal{P}_T}$.
% It suffices to bound $\Lambda_{\mathcal{C}_T}$ and $\Lambda_{\mathcal{B}_T}$ separately.
We start with the easier bound \eqref{eq:single-tree-lac}.

\begin{proposition}
\label{prop:STlac}
There exists a constant $C$ such that for any $n' \in \{1,2,3\}$  
\[
\abs{ \Lambda_{\mathcal{B}_T}(f_1,f_2,f_3) } \le C \abs{ I_T } \Sigma_{n,\infty,f_{n'}}^{\BDR}(T)  \prod_{n \ne n'}   \Sigma_{n,f_n}^{\SUM}(T).
\]
\end{proposition}
\begin{proof} 
By H\"older's inequality in $\R^d$ and the Cauchy-Schwartz inequality in $\mathcal{B}_T$ 
\begin{multline*}
  \abs{ \Lambda_{\mathcal{B}_T}(f_1,f_2,f_3) }  
\le  \sup_{P \in \mathcal{B}_T } \no{1_{P} [ \phi_{P,n'}*f_{n'} ] }_{\infty} \\
\times \prod_{n\ne n'}  \left(\sum_{P \in \mathcal{B}_T} \no{1_{P} [ \phi_{P,n}*f_n ] }_{2}^{2} \right)^{1/2}.
\end{multline*}
This concludes the proof.
\end{proof}

We turn to estimating the form $\Lambda_{\mathcal{B}_T \setminus \mathcal{P}_T}$,
which is the main source of difficulty in the proof.
Here we will need several auxiliary tools,
including Proposition~\ref{globaltreeprop} and some results from \cite{fraccaroli2022phase}. 

\subsection*{Phase space projections}
Define for $j \in \Z$
\[ \mathcal{I}_{T,j} \coloneq \{ I_P  \colon P \in \mathcal{P}_T \setminus \mathcal{B}_T \}   \cap \D_j, \quad E_j^0 \coloneq \bigcup \mathcal{I}_{T,j}.
\]
Define further for each integer $k \ge 1$
\[
\mathcal{I}_{T,j}^k \coloneq \{I\in \D_{j} \colon \rho_I(E_j^0) \le k\}, \quad
E_j^{k} \coloneq \bigcup \mathcal{I}_{T,j}^k.
\]
Finally,
for $\xi \in \R^{3 \times d}$ and $n \in \{1,2,3\}$,
we let $\Mod_{n,\xi}$ be the mapping such that 
\[
\FT (\Mod_{n,\xi} f) (\tau) = \widehat{f}(\tau+\xi_n),
\]
where $\FT$ is the Fourier transform.
We define the phase space localization by using the construction from \cite{fraccaroli2022phase}.

\begin{definition}[Phase plane projection]
Let $v \ge 0$ be an integer, $n \in \{1,2,3\}$ and $T$ be a tree.
Let $h$ be a Schwartz function.
We define $\Pi_{T,n}h = \Mod_{n,-\xi} g$ 
where $g$ is the output of Theorem~1.1 in \cite{fraccaroli2022phase}
based on the input parameter $m = v_n$, 
input function $f = \Mod_{n,\xi}h$,
input cube $U = I_T$, 
and the input $M$ being the family of minimal cubes in $\bigcup_{j \in \Z} \mathcal{I}_{T,j}$ .
\end{definition}
By scaling,
we can now quote the following result from \cite{fraccaroli2022phase}.
\begin{theorem}[Theorem~1.1 in \cite{fraccaroli2022phase}]
\label{thm:ppp}
Let $1 \le p \le \infty$ and  $1/p+1/p'=1$.
Let $\alpha>d$ and $0 \le k \le k_1+4d$.
There exists a constant $C = C(d,\alpha,p,k_0,k_1)$ such that the following holds.

Let $T$ be a tree and fix $n \in \{1,2,3\}$. 
Then for every $j \le j_T$ and $J \in \D_j$  
\begin{equation}   
\label{e:unifgnorm}
\no{ \Pi_{T,n} f}_p \le C \Sigma_{n,p,f}^{\COR}(T) \abs{ I_T }^{1/p}, 
\end{equation}
and
\begin{multline}
\label{e:unif-gtheo}
  \sum_{i \le j_T} \sum_{\substack{ I\in \mathcal{I}_{T,i} \\  I\subset J}}  \sup_{\phi\in \Phi_{n,i-k}^{4\alpha}(\xi)}
\abs{ I }^{1/p'}\no{\rho_I^{-3\alpha} [ \phi *(f-\Pi_{T,n} f) ]}_p  \\
\le C \Sigma_{n,p,f}^{\COR}(T) \abs{ J }.
\end{multline}
For every $j \le j_T$ and $J \in \D_j$ such that $I \nsubset 3J$ for any $I \in \mathcal{I}_{T,j}$ 
\begin{multline}
  \label{e:unigtheo-variant}
\sum_{i \le j_T} \sup_{\substack{I \in \D_i \setminus \mathcal{I}_{T}  \\  I\subset J}} 
\sup_{\psi\in \Psi_{n,i-k}^{4\alpha}(\xi)}
\abs{ I }^{-1/p}\no{\rho_I^{-3\alpha} [ \psi *\Pi_{T,n} f ] }_p \\
 \le C\Sigma_{n,p,f}^{\COR}(T) \no{1_{7I_T} \rho_J^{-\alpha}}_\infty.
\end{multline} 
\end{theorem} 

\begin{proof}[Proof of Proposition~\ref{thm:tree-estimate}]
It remains to prove 
\[
\abs{ \Lambda_{\mathcal{P}_T \setminus \mathcal{B}_T}(f_1,f_2,f_3) } \le C \abs{ I_T } \prod_{n=1}^{3} \Sigma_{n,q_n,f_n}^{\COR}(T).
\]
as by Proposition~\ref{prop:STlac} we already know \eqref{eq:single-tree-lac} to hold.

\subsection*{Core part}
By decomposing $\Lambda_{\mathcal{P}_T \setminus \mathcal{B}_T}$ into $C(d,k_0,k_1)$ many distinct sums, 
we can assume that for each $j \in \Z$,
there is at most one $Q \in \mathcal{W}$ such that 
$Q_P = Q$ and $\abs{ I_P } = 2^{jd}$ for some $P \in \mathcal{P}_T \setminus \mathcal{B}_T$. 
We pick a sequence of functions 
\[
\phi_{j,n} \in \Phi_{n}^{4\alpha}(Q)
\]
such that 
\begin{multline*}
\abs{ \Lambda_{\mathcal{P}_T \setminus \mathcal{B}_T}(f_1,f_2,f_3) } 
\le 
\sum_{j \in \Z}  \max_{P \colon Q_P = Q} \abs{ \int_{\R^{d} } 1_{E_{j}^{1}} (x) \prod_{n=1}^{3} [ \phi_{P,n}*f_n(x) ] \, \diff x } \\
\le  C 
\sum_{j \in \Z}  \int_{\R^{d} } 1_{E_{j}^{1}} (x)  \prod_{n=1}^{3} [ \phi_{j,n}*f_n(x) ] \, \diff x .
\end{multline*}

We define 
\begin{align*}
\Lambda_{\mathcal{C}_T}(f_1,f_2,f_3) & \coloneq \sum_{j \in \Z}  \int_{\R^{d} } 1_{E_{j}^{1}} (x) \prod_{n=1}^{3} [ \phi_{j,n}*f_n(x) ] \, \diff x, \\
\Lambda_{\mathcal{C}_T,c}(f_1,f_2,f_3) & \coloneq \sum_{\substack{j \in \Z \\ E_j^{0} \ne \varnothing}}  \int_{\R^{d} } 1_{(E_{j}^{1})^c} (x)  \prod_{n=1}^{3} [ \phi_{j,n}*f_n(x) ] \, \diff x.
\end{align*}
We compute
\begin{equation} \label{e:tree-proof-three-terms}
\begin{multlined}
\abs{ \Lambda_{\mathcal{C}_T}(f_1,f_2,f_3) } \le \abs{ \Lambda_{\mathcal{C}_T,c}(\Pi_{T,1} f_1,\Pi_{T,2} f_2,\Pi_{T,3} f_3) } \\
+ \abs{ \Lambda_{\mathcal{C}_T}(\Pi_{T,1}f_1,\Pi_{T,2} f_2,\Pi_{T,3}f_3)
+ \Lambda_{\mathcal{C}_T,c}(\Pi_{T,1}f_1,\Pi_{T,2} f_2,\Pi_{T,3}f_3) } \\
+ \abs{ \Lambda_{\mathcal{C}_T}(f_1 - \Pi_{T,1}f_1,\Pi_{T,2} f_2,\Pi_{T,3}f_3) } \\
+ \abs{ \Lambda_{\mathcal{C}_T}(f_1,f_2 - \Pi_{T,2} f_2,\Pi_{T,3}f_3) } \\
+ \abs{ \Lambda_{\mathcal{C}_T}(f_1,f_2,f_3 - \Pi_{T,3}f_3)
} \eqqcolon  \I + \II + \III + \IV + \V
.
\end{multlined}
\end{equation}

For clarity,
we state three auxiliary facts
before estimating the five terms above.
\begin{lemma}
\label{lemma:tree-proof-psi-type}
For each $j \in \Z$, there is $n_j \in \{1,2,3\}$ and coefficients $c_{i,j}$ and functions $\phi_{i,j} \in \Psi_{n,j+i}^{4\alpha}(\xi_T)$ such that 
\[
\phi_{P,n_j} = \sum_{i=-k_1-3d}^{-k_0+1} c_{i,j} \phi_{i,j} , \quad \sum_{i=-k_1-3d}^{-k_0+1} \abs{ c_{i,j} } \le C(d).
\]
\end{lemma}
\begin{proof}
For each $j \in \Z$ and $P \in % \mathcal{C}_T $
\mathcal{P}_T \setminus \mathcal{B}_T$ with $\abs{ I_P } = 2^{jd}$,
we know that $\xi_T \notin 2^{k_0} Q_P$.
Hence there exists at least one $n_j \in \{1,2,3\}$ 
such that $(\xi_T)_{n_j} \notin 2^{k_0} Q_{n_j}$.  
The claim follows from this.
\end{proof}
 
\begin{lemma}
\label{lemma:tree-proof-partition}
Let $\mathcal{A}$ be the set of dyadic cubes $I$ maximal with $\abs{ I } \le \abs{ I_T }$ 
and $J \subset 3I$ for no $J \in \mathcal{I}_T$ with $\abs{ J } \le \abs{ I }$.
Then $\mathcal{A}_j = \{J \in \mathcal{A} \colon \abs{ J } \ge 2^{jd} \}$ is a partition of $\Rd \setminus E_{j}^{1}$.
\end{lemma}
\begin{proof}
Disjointness follows from maximality.
If $x \in \R^{d} \setminus \bigcup \mathcal{A}_{j} $,
then $J \in \mathcal{D}_j$ with $x \in J$ satisfies $3J \supset I$ for some $I \in \mathcal{I}_T$
with $\abs{ I } \le \abs{ J }$.
Then $\widehat{I} \in \mathcal{D}_j$ with $\widehat{I} \supset I$
satisfies $\widehat{I} \in \mathcal{I}_{T,j}$
% and $\widehat{I} \subset J$.
% Hence $I \subset E_{j}^{1}$.
and $J \subset 3 \widehat{I}$.
Hence $J \subset E_{j}^{1}$. The inclusion $\bigcup \mathcal{A}_{j} \subset \R^d \setminus E_{j}^{1}$ follows by definition.
\end{proof}

\begin{lemma}
\label{lemma:tree-proof-singletile}
Let $j \in \Z$ and $J\in \mathcal{D}_j$ be such that $5J \supset I$ for some $I \in \mathcal{I}_{T,j}$.
Then 
\[
% \no{1_{5J}\Pi_{T,n}f}_p 
\no{ 1_{J} [ \phi_{j,n} * \Pi_{T,n} f_n ] }_{q_n} \le C  \abs{ J }^{1/q_n} \Sigma_{n,q_n,f_n}^{\COR}(T)
\] 
\end{lemma}
\begin{proof}
This follows by applying \eqref{e:unif-gtheo} applied to $J$ 
and restricting the sum on the left hand side to a single term as 
\begin{multline*}
\no{ 1_{J} [ \phi_{j,n} * \Pi_{T,n} f_n ] }_{q_n} \le \no{ 1_{J} [ \phi_{j,n} * (\Pi_{T,n} f_n - f_n ) ] }_{q_n}  + \no{ 1_{J} [ \phi_{j,n} * f_n ]}_{q_n}  \\
\le C \abs{ J }^{1/q_{n}} \Sigma_{n,q_n,f_n}^{\COR}(T).
\end{multline*} 
\end{proof}

Now we can estimate the five terms in \eqref{e:tree-proof-three-terms}.
To estimate $\I$,
we recall that for each $j \in \Z$,
there exists $n_j \in \{1,2,3\}$ as in Lemma~\ref{lemma:tree-proof-psi-type}.
We fix $n_j$ to be one of them so that
the three sets $\mathcal{N}_{n} = \{j \in \Z: \ E_{j}^{1} \ne \varnothing, \ n_j = n \}$
partition the subset of $\Z$ appearing in the definition of $\I$.
Then 
\begin{multline*}
  \I = \abs{ \Lambda_{\mathcal{C}_T,c}(\Pi_{T,1}f_1,\Pi_{T,2} f_2,\Pi_{T,3}f_3) } \\
\le 
\sum_{\nu=1}^{3} 
\int_{\Rd} \sum_{j \in \mathcal{N}_\nu}  1_{(E_{j}^{1})^{c}} \prod_{n=1}^{3} |\phi_{j,n}*\Pi_{T,n}f_n(x)| \, \diff x \\
\le \sum_{\nu=1}^{3} \left(\prod_{n \ne \nu} \no{M_{\HL}\Pi_{T,n}f_n}_{q_n} \right) \nos{\sum_{j \in \mathcal{N}_\nu} 1_{(E_{j}^{1})^{c}}|\phi_{j,\nu}*\Pi_{T,\nu}f_\nu|}_{q_\nu}
\end{multline*}
where $M_{\HL}$ is the Hardy--Littlewood maximal function.
By the maximal function theorem and \eqref{e:unifgnorm} from Theorem \ref{thm:ppp}
\[
\no{M_{\HL}\Pi_{T,n}f_n}_{q_n} \le C |I_T|^{1/q_n} \Sigma_{n,q_n,f_n}^{\COR}(T).
\] 
By Lemma \ref{lemma:tree-proof-partition}
and Minkowski's inequality
\begin{multline*}
\nos{\sum_{j \in \mathcal{N}_\nu} 1_{(E_{j}^{1})^{c}}|\phi_{j,\nu}*\Pi_{T,\nu}f_\nu|}_{q_\nu}^{q_{\nu}} \\
= \sum_{J \in A} \nos{ 1_{J} \sum_{j \in \mathcal{N}_\nu} \sum_{\substack{I \in \mathcal{D}_j \setminus \mathcal{I}_{T,j}^{1}    }}  1_{I}|\phi_{j,\nu}*\Pi_{T,\nu}f_\nu|}_{q_\nu}^{q_{\nu}} \\
\le \sum_{J \in A} \left(\sum_{j \in \mathcal{N}_\nu} \left(\sum_{\substack{I \in \mathcal{D}_j \setminus \mathcal{I}_{T,j}^{1}  \\ I \subset J }} \no{1_{I} [\phi_{j,\nu}*\Pi_{T,\nu}f_\nu] }_{q_{\nu}}^{q_{\nu}} \right)^{1/q_\nu}    \right)^{q_\nu} .
\end{multline*}
By Lemma~\ref{lemma:tree-proof-psi-type}
and \eqref{e:unigtheo-variant} from Theorem \ref{thm:ppp}
\begin{multline*}
\sum_{j \in \mathcal{N}_\nu} \left(\sum_{\substack{I \in \mathcal{D}_j \setminus \mathcal{I}_{T,j}^{1}  \\ I \subset J }} \no{1_{I} [\phi_{j,\nu}*\Pi_{T,\nu}f_\nu] }_{q_{\nu}}^{q_{\nu}} \right)^{1/q_\nu} \\
\le |J|^{1/q_{n}} \sum_{i=-k_1-3d}^{-k_0+1}  \sum_{j \in \mathcal{N}_\nu} 
\sup_{\substack{I \in \mathcal{D}_j \setminus \mathcal{I}_{T,j}^{1} \\ I \subset J}} \sup_{\psi \in \Psi_{\nu,j+i}^{4\alpha}(\xi_T)}  \frac{ \no{ 1_{I} [ \psi * \Pi_{T,\nu} f_{\nu} ] }_{q_{\nu}}}{\abs{I}^{1/q_{\nu}}} \\
\le C |J|^{1/q_{\nu}} \Sigma_{\nu,q_\nu,f_\nu}^{\COR}(T) \no{1_{7I_T} \rho_{J}^{-\alpha}  }_{\infty}.
\end{multline*}
Summing the $q_{\nu}$th power over $J$ concludes the proof.

To estimate 
\[
\II = \abs{ \Lambda_{\mathcal{C}_T}(\Pi_{T,1}f_1,\Pi_{T,2} f_2,\Pi_{T,3}f_3) + \Lambda_{\mathcal{C}_T,c}(\Pi_{T,1}f_1,\Pi_{T,2} f_2,\Pi_{T,3}f_3) },
\]
it suffices to apply 
the global paraproduct estimate Proposition~\ref{globaltreeprop} and the $L^{p}$ estimate for the phase space projection, \eqref{e:unifgnorm} in Theorem~\ref{thm:ppp}.
The desired bound follows.

We move to estimate $\III + \IV + \V$.
Note that for $n \in \{1,2,3\}$ and $J \in % \mathcal{I}_{T,j}
\mathcal{I}_{T,j}^{1}$ by definition of $ \Sigma_{n,p,f}^{\COR}(T)$
\[
\no{ 1_{J} [ \phi_{j,n} * f_{n} ] }_{q_{n}} \le \abs{ J }^{1/q_{n}}  \Sigma_{n,q_n,f_n}^{\COR}(T)
\]
and further by Lemma~\ref{lemma:tree-proof-singletile}
\[
\no{ 1_{J} [ \phi_{j,n} * \Pi_{T,n} f_n ] }_{q_n} \le C \abs{ J }^{1/q_{n}} \Sigma_{n,q_n,f_n}^{\COR}(T).
\] 
By these estimates and H\"older's inequality  
\begin{multline*}
\III + \IV + \V \le C \max_{n \in \{1,2,3\}} \Bigg\{ % \left\{ 
\left(\prod_{n' \ne n} \Sigma_{n',q_{n'},f_{n'}}^{\COR}(T) \right) \\
\times \sum_{j \in \Z} \sum_{J \in \mathcal{I}_{T,j}^{1}} \abs{ J }^{1-1/q_n} % \Sigma_{n,p,f_{n}}^{\COR}(T)
\no{ 1_{J} [ \phi_{j,n} * ( f_{n} - \Pi_{T,n} f_n ) ] }_{q_{n}}  % \right\},
\Bigg\},
\end{multline*} 
from which the claim follows by \eqref{e:unif-gtheo} of Theorem~\ref{thm:ppp}.
\end{proof}

%\changelocaltocdepth{2}

\section{Proof of  Proposition~\ref{almostorthogonal}. Tree selection}
\label{s:almostorthogonal}

We start by defining two auxiliary sizes that are needed to complement those in Definition \ref{def:main-sizes}.

\begin{definition}
\label{def:aux-sizes}
Under the set-up of Definition \ref{def:main-sizes},
define
\begin{align*} 
\Sigma_{n,p,f}^{\BDR,\TOP}(T)  &=  
\sup_{\substack{P \in \mathcal{B}_T \\ I_P = I_T}}
\sup_{\phi \in \Phi_{n}^{4\alpha}(Q_P)}  
\frac{\no{\rho_{I_T}^{-\alpha} [\phi*f]}_p}{\abs{ I_T }^{1/p}}     , \\ 
\Sigma_{n,p,f}^{\COR,\TOP}(T)  &=     \sup_{\phi \in \Phi_{n,j_T-k_1-5d}^{4\alpha}(\xi_T)}  
\frac{\no{\rho_I^{-\alpha} [\phi* f]}_p}{\abs{ I_T }^{1/p}} . \\
\end{align*}
\end{definition}

We formalize the idea of greedy selection by stating the following definition. 
\begin{definition}[Selection]
%\label{ref:stoppedselection}
Let $\mathcal{V}$ be a finite set of multitiles.
Let $\mathcal{T}$ be the family of all trees in any of the subsets of % $\mathcal{V}$
$\mathcal{V}$.  
Let $S$ be a positive integer.
A selection is a mapping $\sigma \colon \{1,\ldots,S\}  \to \mathcal{T} $ 
such that 
\begin{itemize} 
  \item $\sigma(1)$ is a tree in % $\mathcal{V}$
$\mathcal{V}_{\sigma(1)} = \mathcal{V}$;
  \item $\sigma(i+1)$ is a tree in % $\mathcal{V}_{\sigma(i)} \setminus \mathcal{P}_{\sigma(i)}$
$\mathcal{V}_{\sigma(i+1)} = \mathcal{V}_{\sigma(i)} \setminus \mathcal{P}_{\sigma(i)}$ for all $i \in \{1, \dots, S-1\}$.  
\end{itemize}
\end{definition}

To prove Proposition~\ref{almostorthogonal},
we will construct several selections over the initial set of multitiles.
We first show that selections based on top size defined above have good orthogonality properties 
and as a second step we show that convexity properties allow us to infer estimates for main sizes of Definition \ref{def:main-sizes} from those for the auxiliary top sizes of Definition \ref{def:aux-sizes}.
There will be three different selection processes.
The first selection serves to identify the trees with large core size.
The following proposition shows that they have controlled overlap.  

\begin{proposition} 
\label{prop:ao-core}
There exists a constant $C$ such that the following holds.

Let $D > 1$.
Let $f \in L^{2}(\R^{d})$.
Let $\mathcal{V}$ be a finite set of multitiles and let $\sigma$ be a selection in $\mathcal{V}$.
Let $M > 0$.  
Assume the following properties of the selection.
\begin{itemize}
  \item If $I_i$ is the top cube of $\sigma(i)$ and if $I_{i+1}$ is the top cube of $\sigma(i+1)$,   
then $\abs{I_{i+1}} \le \abs{I_{i}}$ for all $i \in \{1, \dots, S-1\}$.
  \item For each $i \in \{1, \dots, S\}$, there exists $A_i \in \mathcal{P}_{\sigma(i)}$ with $2^{k_1+1}Q_{A_i} \ni \xi_{\sigma(i)}$.
  \item For each $i \in \{1, \dots, S\}$, it holds $M \le (\Sigma_{n,2,f}^{\COR,\TOP} \circ \sigma) (i) \le  DM$.
\end{itemize}
Then
\[
\left(\sum_{i = 1}^S M^{2} \abs{I_{\sigma(i)}} \right)^{1/2} \le C D \no{f}_{2} .
\]
\end{proposition}

The next selection serves to remove the trees that contain a lacunary multitile, 
not treated by the core size,
that however happens to gives a large contribution.

\begin{proposition} 
\label{prop:ao-single-lac}
There exists a constant $C$ such that the following holds.

Let $D > 1$.
Let $f \in L^{2}(\R^{d})$.
Let $\mathcal{V}$ be a finite set of multitiles and let $\sigma$ be a selection in $\mathcal{V}$.
Let $M > 0$.  
Assume the following properties of the selection.
\begin{itemize}
  \item If $I_i$ is the top cube of $\sigma(i)$ and if $I_{i+1}$ is the top cube of $\sigma(i+1)$,   
then $\abs{I_{i+1}} \le \abs{I_{i}}$ for all $i \in \{1, \dots, S-1\}$.
  \item For each $i \in \{1, \dots, S\}$, it holds $M \le (\Sigma_{n,2,f}^{\BDR,\TOP} \circ \sigma) (i) \le  DM$.
\end{itemize}
Then
\[
\left(\sum_{i=1}^S M^{2} \abs{I_{\sigma(i)}} \right)^{1/2} \le C D \no{f}_{2} .
\]
\end{proposition}

The third selection removes the trees whose boundaries are contributing a lot to the right hand side of Proposition~\ref{thm:tree-estimate}.
While the choice order of the previous selections was based on metric geometry,
only using the size of the top cube,
the treatise of the lacunary parts of the trees requires us to 
carry out a cone decomposition and consider an order of selection based on that.

\begin{proposition} 
\label{prop:ao-sum-lac}
There exists a constant $C$ such that the following holds.

Let $D > 1$ and let $e$ be a unit vector orthogonal to $d-1$ coordinate axes.
Let $f \in L^{2}(\R^{d})$.
Let $\mathcal{V}$ be a finite set of multitiles and let $\sigma$ be a selection in $\mathcal{V}$.
Let $M > 0$.  
For each $i \in \{1, \dots, S\}$,
denote 
\[
C_{i} = \{\xi \in \R^{d} \colon \abs{\xi-\xi_{\sigma(i)}} \le  2(\xi - \xi_{\sigma(i)}) \cdot e \}
\]
and let $\mu_i \in M_n(\xi_{\sigma(i)},C_i)$.
Assume the following properties of the selection.
\begin{itemize}
  \item For all $i \in \{1, \dots, S-1\}$, assume that $\xi_{\sigma(i)}\cdot e \ge \xi_{\sigma(i+1)}\cdot e$.
  \item For each $i \in \{1, \dots, S\}$, it holds $M \le ( \Sigma_{n,\mu_i*f}^{\SUM} \circ \sigma ) (i)$ and $ ( \Sigma_{n,2,\mu_i*f}^{\BDR} \circ \sigma ) (i) \le DM$.
\end{itemize}
Then
\[
\left(\sum_{i = 1}^S M^{2} \abs{I_{\sigma(i)}} \right)^{1/2} \le C D \no{f}_{2} .
\]
\end{proposition}

To apply the propositions stated above,
we still have to solve the discrepancy between the definitions of sizes in Definition \ref{def:main-sizes} 
and Definition \ref{def:aux-sizes}.
This is the content of the last proposition of this section.
We need one more definition.

\begin{definition}[Convex collection]
A finite family of multitiles $\mathcal{V}$ is a convex collection 
if for any tree $T$ on $\mathcal{V}$ and 
\[
j_{min} = \min_{P \in \mathcal{P}_T} \log_2 \abs{I_P}^{1/d}
\]
the condition $j \in \Z \cap \{i \colon j_{min} \le i \le j_T\}$ implies that there exist $P \in \mathcal{P}_{T}$ with $\abs{I_P} = 2^{jd}$ and the condition that $2^{k_1+1}Q_P \ni \xi_T$ for some $P \in \mathcal{P}_T$ implies $2^{k_1+1}Q_{P'}$ for a $P' \in \mathcal{P}_T$ with $I_{P'} = I_{T}$.
\end{definition}

For the purpose of the proof of our main theorem,
the convex collections are the only ones that matter.
The importance of the convex collections lies in the fact that 
every tree on a convex collection has a subtree whose size is attained by one of its top multitiles.

% \begin{definition}[Unhollow trees]
% Let $T$ be a tree. We set 
% \[
% \Theta(T) = \begin{cases}
%   1, \quad \trm{if there exists $P \in \mathcal{P}_T$ with $2^{k_1+1}Q_P \ni \xi_{T}$}. \\
%   0, \quad \trm{otherwise}.
% \end{cases}
% \]  
% \end{definition}

Moreover, for a tree $T$ we set
\[
\Theta(T) = \begin{cases}
1, \quad \trm{if there exists $P \in \mathcal{P}_T$ with $2^{k_1+1}Q_P \ni \xi_{T}$}. \\
0, \quad \trm{otherwise}.
\end{cases}
\]

\begin{proposition}
\label{prop:induction-step}
Let $\mathcal{V}$ be a convex family.
Let $\{e_{\delta} \colon 1 \le \delta \le 2d  \}$ be the unit vectors orthogonal to the $(d-1)$-dimensional coordinate hyperplanes.
Let 
\[
C_{e} = \{\xi \in \R^{d} \colon \abs{\xi-\xi_{\sigma(i)}} \le  2(\xi - \xi_{\sigma(i)}) \cdot e \}.
\]
Let $\mu^{\delta,n} \in M_n(C_{e_{\delta}})$ with 
\[
\sum_{\delta=1}^{2d} \widehat{\mu}^{\delta,n}(\xi) = 1, \quad \xi \ne 0.
\] 
For a tree $T$ on $\mathcal{V}$,
we set $\widehat{\mu}^{\delta}_{T,n}(\xi) = \widehat{\mu}^{\delta,n}(\xi-(\xi_T)_n)$.

Let $M \in \Z$ be such that for all trees $T$ on $\mathcal{V}$
\begin{equation*}
\label{eq:sele-1}
% \max_{\substack{1 \leq \delta \leq 2d \\ n \in \{1,2,3\}}} \max \left\{ \frac{\Sigma_{n,2,f_n}^{\BDR}(T)}{\no{f_n}_{2}}   ,  \frac{\Sigma_{n,\mu_{T}^{\delta,n} * f_n}^{\SUM}(T)}{\no{f_n}_{2}}   ,  \frac{\Theta(T)\Sigma_{n,2,f_n}^{\COR}(T)}{\no{f_n}_{2}} \right\} 
\max_{\substack{1 \leq \delta \leq 2d \\ n \in \{1,2,3\}}} \max \left\{ \Sigma_{n,2,f_n}^{\BDR}(T)  ,  \Sigma_{n,\mu_{T}^{\delta,n} * f_n}^{\SUM}(T) , \Theta(T)\Sigma_{n,2,f_n}^{\COR}(T) \right\} \le 2^{M/2} \no{f_n}_{2}  . 
\end{equation*}
Then there exists a selection $\sigma$ on $\mathcal{V}$ such that 
\[
\widetilde{\mathcal{V}}_{M-1} = \mathcal{V} \setminus \bigcup_{i=1}^S \mathcal{P}_{\sigma(i)}
\]
is a convex family such that for all trees on $\widetilde{\mathcal{V}}_{M-1}$
\begin{multline}
\label{eq:sele-2}
% \max_{\substack{1 \leq \delta \leq 2d \\ n \in \{1,2,3\}}} \max \left\{ \frac{\Sigma_{n,2,f_n}^{\BDR}(T)}{\no{f_n}_{2}}   ,  \frac{\Sigma_{n,\mu_{T}^{\delta,n} * f_n}^{\SUM}(T)}{\no{f_n}_{2}}   ,  \frac{\Theta(T)\Sigma_{n,2,f_n}^{\COR}(T)}{\no{f_n}_{2}} \right\} 
\max_{\substack{1 \leq \delta \leq 2d \\ n \in \{1,2,3\}}} \max \left\{ \Sigma_{n,2,f_n}^{\BDR}(T) , \Sigma_{n,\mu_{T}^{\delta,n} * f_n}^{\SUM}(T)  ,  \Theta(T)\Sigma_{n,2,f_n}^{\COR}(T) \right\} 
\\
\le 2^{(M-10d)/2} \no{f_n}_{2}
\end{multline}
and  
\begin{equation}
\label{e:induction-step-overlap}
\sum_{i=1}^S 2^{M} \abs{ I_{\sigma(i)} } \lesssim 1.
\end{equation}
\end{proposition}

\subsection*{Auxiliary propositions for almost orthogonality}
In this subsection,
we prove two additional estimates that are needed in the proofs of Propositions \ref{prop:ao-core}, \ref{prop:ao-single-lac} and \ref{prop:ao-sum-lac}. 

\begin{proposition}
\label{prop:commute-phi-rho}
Let $\alpha > 2d$.
There exists a constant $C$ such that the following holds.

Let $j \in \Z$, $k \ge 0$, $\xi \in \R^{d}$ and $f \in L^{\infty}(\R^{d})$.
Let $\varphi \in \Phi_{n,j-k}^{4\alpha}(\xi)$ and $I$ be a cube with $\abs{I} = 2^{jd}$.
Denote by $\mathcal{M}_{\HL}$ the Hardy--Littlewood maximal function.
Then for all $x \in \R^{d}$
\[
\abs{ \varphi*(\rho_{I}^{-\alpha} f)(x) } \le C  \rho_{I}(x)^{-\alpha} \mathcal{M}_{\HL}f(x).
\]   
\end{proposition}
\begin{proof}
As for $j' \leq j$ we have $\rho_{[0,2^{j'})^d} \geq \rho_{[0,2^{j})^d}$, then for any $\varphi \in \Phi_{n,j-k}^{4\alpha}(\xi)$ and $x \in \R^{d}$,
it holds
\[
\abs{ \rho_{[0,2^{j-k})^{d}}^{\alpha}(x) \varphi(x) }  \le 2^{(j-k+v_n)d}  \rho_{[0,2^{j-k-v_n})^{d}}^{-3\alpha}(x).
\] 
% As $\rho_{[0,2^{j-k})}^{-3\alpha} \le \rho_{[0,2^{j})}^{-3\alpha}$,
We also have  
\[
\rho_{[0,2^{j})^d}^{-\alpha}(x-y) \rho_{I}^{-\alpha}(y) \le C \rho_{I}^{-\alpha}(x) 
\]
for all $x,y \in \R^{d}$.
Indeed, if $2 \rho_{I}(y) \ge  \rho_{I}(x)$, 
this is clear,
and if $2 \rho_{I}(y) \le \rho_{I}(x)$, 
then
\[
\rho_{[0,2^{j})^d}(x-y) \ge \rho_{I}(x) - \rho_{I}(y) \ge \frac{\rho_{I}(x)}{2} .
\]
In conclusion, 
\begin{multline*}
\abs{ \varphi*(\rho_{I}^{-\alpha} f)(x) } \\
= 
\abs{ \int_{\R^d} \rho^{\alpha}_{[0,2^{j})^{d}}(x-y)\varphi(x-y)  \rho^{-\alpha}_{[0,2^{j})^{d}}(x-y)\rho^{-\alpha}_{I}(y) f(y) \, \diff y}  \\
\le C \rho_{I}(x)^{-\alpha} [ 2^{(j-k+v_n)d} \rho^{-3\alpha}_{[0,2^{j-k-v_n})^{d}}*(\rho_{I}^{-\alpha} \abs{f})] (x).
\end{multline*}
This concludes the proof.
\end{proof}

The second auxiliary proposition is essentially a restatement of
Lemmata~5.1,~5.2 and~5.3 in \cite{MR1979774}.
Also this estimate is needed in the proofs of all the Propositions~\ref{prop:ao-core},~\ref{prop:ao-single-lac} and~\ref{prop:ao-sum-lac}.
\begin{proposition}
\label{prop:overlap}
Let $A_1$ be a positive constant and let $\alpha  > d$.
Then there exists a constant $A_2$ such that the following holds.

Let $J \in \mathcal{D}$.
Let $\mathcal{I} \subset \mathcal{D}$ be a family of cubes satisfying 
\[
\sum_{\substack{I \in \mathcal{I} \\ I \subset I'}} \abs{I} \le A_1 \abs{I'}
\]
for all cubes $I'$ and $|I| \le |J|$ for all $I \in \mathcal{I}$.
For each $I \in \mathcal{I}$,
let $g_I \in L^{2}(\R)$ be given. 
Then 
\begin{equation}
\label{eq:prop-bctr}
\nos{ \rho_{J}^{-\alpha} \sum_{I \in \mathcal{I}} \abs{I}^{1/2} g_I \rho_{I}^{- \alpha} }_2
\le A_2 \abs{J}^{1/2}  \sup_{I \in \mathcal{I}} \no{g_I}_{2}  .
\end{equation} 
\end{proposition}

\begin{proof}
We first prove the reminiscent inequality 
\begin{equation}
\label{eq:in-best-const-trick}
\nos{\sum_{I \in \mathcal{I}} \abs{I}^{1/2} g_I 1_{D I} }_2
\le 2D^{d} \sqrt{5A_1}  \sup_{I \in \mathcal{I}} \no{g_I}_{2}  \left(  \sum_{I \in \mathcal{I}} \abs{I} \right)^{1/2}
\end{equation} 
for all odd numbers $D \ge 3$.
Here the non-local cut-off functions are replaced by sharp cut-off functions.

Fix a family $\mathcal{I}$ and the corresponding functions $g_I$.
% Let $A$ be the sharp constant for the inequality \eqref{eq:in-best-const-trick} 
% when considered over all finite subfamilies of $\mathcal{I}$.
% Let $\mathcal{I'} \subset \mathcal{I}$ be finite.
Let $\mathcal{I'} \subset \mathcal{I}$ be finite.
Let $A$ be the sharp constant for the inequality \eqref{eq:in-best-const-trick} 
when considered over all finite subfamilies of $\mathcal{I}'$.
Then 
\begin{multline*}
\nos{\sum_{I \in \mathcal{I}'  } \abs{I}^{1/2} g_I 1_{D I} }_2^{2}
\le 
2 \sum_{I \in \mathcal{I}'} 
\sum_{\substack{J \in \mathcal{I}'\\ DJ \subset 5DI}}
\la{\abs{I}^{1/2} g_I 1_{D I} , \abs{J}^{1/2} g_J 1_{D J}} \\
\le 
2   \sum_{I \in \mathcal{I}'} 
\abs{I}^{1/2} \no{g_I}_2 \nos{ \sum_{\substack{J \in \mathcal{I}'\\ DJ \subset 5DI}} \abs{J}^{1/2} g_J 1_{D J} }_2 \\
\le 
 2 A D^{d/2} \sup_{I \in \mathcal{I}'}  \no{g_I}_2 ^{2} \sum_{I \in \mathcal{I}'} 
\abs{I}^{1/2}   \left( \sum_{\substack{J \in \mathcal{I}'\\  J \subset 5DI}} \abs{J} \right)^{1/2} \\
\le 2 \sqrt{5 A_1}   A D^{d} \sup_{I \in \mathcal{I}'}  \no{g_I}_2^{2} \left( \sum_{I \in \mathcal{I}' } \abs{I} \right) .
\end{multline*}
Consequently $A  \le 2D^{d} \sqrt{5A_1} $.
As this constant is independent of $\mathcal{I}'$ and the functions $g_I$,
the proof of \eqref{eq:in-best-const-trick} is complete.

To prove \eqref{eq:prop-bctr},
we write 
\[
% \rho_{I}^{-2\alpha} \le \sum_{k=1}^{\infty} k^{-2\alpha} 1_{(2k-1)I}
\rho_{I}^{-\alpha} \le \sum_{k=1}^{\infty} k^{-\alpha} 1_{(2k-1)I}
\]
and 
\[
\rho_{J}^{-\alpha} \le 1_{J} + \sum_{l=1}^{\infty} % (2l-1)^{-\alpha} 
l^{-\alpha} 1_{ (2l+1) J \setminus (2l-1) J}.
\] 
Set $\mathcal{I}_{k,l} = \{I \in \mathcal{I} \colon (2k-1)I \cap (2l+1) J \ne \varnothing \}$.
Then 
\begin{multline*}
\nos{ \rho_{J}^{-\alpha} \sum_{I \in \mathcal{I}} \abs{I}^{1/2} g_I \rho_{I}^{-\alpha} }_2  \\
\le \sum_{k,l = 1}^{\infty} % (2k-1)^{-2\alpha} (2l-1)^{-\alpha} 
k^{-\alpha} l^{-\alpha} \nos{ 1_{ (2l+1)J \setminus (2l-1) J} \sum_{I \in \mathcal{I}_{k,l}} \abs{I}^{1/2} g_I 1_{(2k-1)I} }_2 % .
\\
\leq \sum_{k,l = 1}^{\infty} k^{-\alpha} l^{-\alpha} \nos{ 1_{ (2 \max \{ l, k \} + 2) J } \sum_{I \in \mathcal{I}_{k,l}} \abs{I}^{1/2} g_I 1_{(2k-1)I} }_2 \\
\lesssim \sum_{k,l = 1}^{\infty} 2^{-d/2} k^{-\alpha - d/2} l^{-\alpha} \abs{ (2 \max \{ l, k \} + 2) J }^{1/2} \lesssim \abs{J}^{1/2} .
\end{multline*}
\end{proof}

Finally,
we state as a separate proposition the obvious fact that elements of $\mathcal{W}$
with overlapping projections 
are close to each other in the product space too,
something that is a direct consequence of the defining inequalities of $\mathcal{W}$.
\begin{proposition}
\label{prop:almorth-projections}
Let $a\ge 0$.
Let $Q,Q' \in \mathcal{W}$ satisfy $\abs{Q} \geq \abs{Q'}$.
Assume that for some $n \in \{1,2,3\}$ there exist
\[
\xi \in 2^{a} Q_n \cap 2^{a}Q_n'  , \quad
\eta \in 2^{a} Q \cap \Gamma  , \quad
\zeta \in 2^{a} Q' \cap \Gamma  .
\]
Then $2^{a+4} Q \supset Q'$.
\end{proposition}

\begin{proof}
First we note that the projection $\mathfrak{P} \colon \Gamma \to \Rd$ defined through $\mathfrak{P}\xi = \xi_n$ is a bijection.
This follows from the regularity of $L$ and the fact $\Gamma = \{L(\tau,\tau,\tau) \colon \tau \in \R^{d}\}$.
Consider the metrics 
\begin{align*}
  \dist_{\FULL}(\xi,\eta) &= \inf\{r \colon \eta \in Q(\xi,r) \},\\
  \dist_n(\xi_n,\eta_n) &= \inf\{r \colon \eta_n \in Q_n(\xi_n,r) \}
\end{align*}
The left inverse $\mathfrak{P}^{-1}$ is a $2$-Lipschitz mapping $(\Rd, \dist_n) \to (\Gamma, \dist_{\FULL})$
following directly from the definition of the metrics.
We infer that 
\begin{align*}
\dist_{\FULL}(\mathfrak{P}^{-1}\xi,\mathfrak{P}^{-1}\mathfrak{P} \eta) &\le 2 \dist_n (\xi,\mathfrak{P}\eta), \\
\dist_{\FULL}(\mathfrak{P}^{-1}\xi,\mathfrak{P}^{-1}\mathfrak{P} \zeta) &\le 2 \dist_n (\xi,\mathfrak{P}\zeta)   
\end{align*}
so that 
\[
\dist_{\FULL}(\mathfrak{P}^{-1}\mathfrak{P} \zeta,\mathfrak{P}^{-1}\mathfrak{P} \eta) \le 2^{a+2} \diam Q_n.
\]
where the diameter $\diam$ is computed with respect to $d_n$.
As $\abs{Q} \ge \abs{Q'}$, we conclude $2^{a+3} Q \cap Q' \ne \varnothing$
and the claim follows.
\end{proof}

As an immediate corollary of Proposition~\ref{prop:almorth-projections},
we conclude that the multitiles $P \in \mathcal{B}_T$ have all their frequency projections supported far from the projections of the top frequency.
This will imply important $L^{2}$ orthogonality properties for the sum size. 

\begin{proposition}
\label{prop:lacunary-trivial}
Given $P \in \mathcal{B}_T$ with $\abs{I_P} = 2^{jd}$ for some $j \in \Z$,
we have for all $n \in \{1,2,3\}$
\begin{equation*}
\label{e:lacunary-trivial}
(Q_{P})_n \subset \Rd \setminus Q_n(\xi_T,2^{-j-k_2}).
\end{equation*}
\end{proposition}

\begin{proof}
Denote $Q = Q_P$.
By construction $ 2^{k_0+1} Q \cap \Gamma \ne \varnothing$.
On the other hand,
as $P \in \mathcal{B}_T$,
we know that $\xi_T \notin 2^{k_1+1} Q$.
%Let $c$ be the center of $Q$ and 
Set $Q' = Q(\xi_T,2^{-j+k_1/50})$.
Then $Q' \cap  Q = \varnothing$.
It follows by Proposition~\ref{prop:almorth-projections} that 
$Q'_{n} \cap Q_{n} = \varnothing$.
\end{proof}

\changelocaltocdepth{2}

\subsection{Proof of Proposition~\ref{prop:ao-core}. Core size} 

For each $i \in \{1, \dots, S\}$,
we find 
\[
\phi_i \in \Phi_{n,j_{\sigma(i)}-k_1-5d}^{4\alpha}(\xi_{\sigma(i)}) 
\]
such that 
\[
c_i =  \no{\rho_{I_{\sigma(i)}}^{-\alpha} [ \phi_i* f ]}_2 % \sim  M\sqrt{\abs{ I_{\sigma(i)} }}.
, \quad M\sqrt{\abs{ I_{\sigma(i)} }} \leq c_i \leq D M\sqrt{\abs{ I_{\sigma(i)} }}.
\]
Let
\[
g_i = \frac{\rho_{I_{\sigma(i)}}^{-\alpha} [ \phi_i* f ]}{\no{\rho_{I_{\sigma(i)}}^{-\alpha} [ \phi_i * f ]}_2}.
\]
Now
\begin{multline*}
\sum_{i = 1}^S M^{2} \abs{ I_{\sigma(i)} } \lesssim 
\sum_{i = 1}^S c_i \abs{ I_{\sigma(i)} }^{1/2} \la{f, \bar \phi_i* (\rho_{I_{\sigma(i)}}^{-\alpha} g_i ) } \\
\le \no{f}_{2} \nos{\sum_{i = 1}^S c_i \abs{ I_{\sigma(i)} }^{1/2} [ \phi_i* (\rho_{I_{\sigma(i)}}^{-\alpha} g_i ) ] }_2.
\end{multline*}
Expanding the square and using the symmetry,
we obtain
\begin{equation} \label{eq:tree-sel-2}
\begin{multlined}
\nos{ \sum_{i = 1}^S \abs{ I_{\sigma(i)} }^{1/2} c_i [ \phi_i* (\rho_{I_T}^{-\alpha} g_i ) ] }_2^{2} \\
\lesssim D^2 M^{2} \sum_{i=1}^S \abs{ I_{\sigma(i)} }^{1/2} \sum_{l = i}^S 
\abs{ I_{\sigma(l)} }^{1/2} \la{ \rho_{I_{\sigma(l)}}^{-\alpha} g_{l} , \phi_{l} * [ \phi_i *(\rho_{I_{\sigma(i)}}^{-\alpha} g_i)]}.
\end{multlined}
\end{equation}

Let 
\[
\mathcal{A}_i  = \{ l \in \{i, \dots, S\}  \colon \supp \widehat{\phi}_{l} \cap \supp \widehat{\phi}_{i} \ne \varnothing  \}
\]
By Proposition~\ref{prop:commute-phi-rho},
for all $l \in \mathcal{A}_i$
\[
\phi_{l} * [ \phi_i *(\rho_{I_{\sigma(i)}}^{-\alpha} g_i) ] \le C \rho_{I_{\sigma(i)}}^{-\alpha} \mathcal{M}_{\HL} g_i 
\]
where $\mathcal{M}_{\HL}$ is the Hardy--Littlewood maximal function.
Using this estimate,
Cauchy--Schwarz inequality and 
the Hardy--Littlewood maximal function theorem, 
we bound the right hand side of \eqref{eq:tree-sel-2} 
by 
\[
C D^2 M^{2} \sum_{i = 1}^S \abs{ I_{\sigma(i)} } \nos{\frac{1}{ \abs{ I_{\sigma(i)} }^{1/2}} \sum_{l \in \mathcal{A}_i} 
\abs{ I_{\sigma(l)} }^{1/2} g_{l} \rho_{I_{\sigma(l)}}^{-\alpha} \rho_{I_{\sigma(i)}}^{-\alpha} }_{2}.
\]
By hypothesis,
for each $l \in \{1, \dots, S\}$ there exists a top multitile $A_l \in \mathcal{P}_{\sigma(l)}$ with $\xi_{\sigma(l)} \in 2^{k_1+1} Q_{A_l}$.
Hence given $l,j \in \mathcal{A}_i$ with $l > j$ we have
\begin{equation*}
( 2 \supp \widehat{\phi}_{l} ) \cap ( 2 \supp \widehat{\phi}_{j} ) \ne \varnothing.   
\end{equation*}
Therefore, we have
\begin{center}
$( 2 \supp \widehat{\phi}_{l} ) \cap ( 2 \supp \widehat{\phi}_{j} ) \ne \varnothing $ and $I_{\sigma(l)} \cap I_{\sigma(j)} \ne \varnothing$
\end{center}
only if $ \abs{ I_{\sigma(l)} } \gtrsim \abs{ I_{\sigma(j)} }$ 
as otherwise Proposition~\ref{prop:almorth-projections} would imply 
\[
2^{k_2+1}Q_{A_l} \supset 2^{k_1+5} Q_{A_{l}} \supset 2^{k_1+1}Q_{A_{j}} \ni \xi_{\sigma(j)},
\]
which in turn would contradict $A_l \in \mathcal{V}_{\sigma(l)}$. 
By the definition of selection $\abs{ I_{\sigma(l)} } \le \abs{ I_{\sigma(j)} }$. 
Moreover, for every fixed $ I_{\sigma(j)}$ there are only up to $C(d,k_0)$ elements $l \in \mathcal{A}_i$ such that $I_{\sigma(j)} = I_{\sigma(l)}$, so we can conclude that for any $i \in \{1,\dots,S\}$

\begin{equation*}
\nos{ \sum_{l \in \mathcal{A}_i} 1_{I_{\sigma(l)}} }_{\infty}
\lesssim 1,
\end{equation*} 
hence $\{I_{\sigma(l)} \colon l \in \mathcal{A}_i \}$ is a Carleson family. By Proposition~\ref{prop:overlap}
\[
\nos{\frac{1}{\abs{ I_{\sigma(i)} }^{1/2}} \sum_{l \in \mathcal{A}_i} 
\abs{ I_{\sigma(l)} }^{1/2} g_{l} \rho_{I_{\sigma(l)}}^{-\alpha} \rho_{I_{\sigma(i)}}^{-\alpha} }_{2} \lesssim 1 
\]
and we have shown the claim for the sum over all $i \in \{1, \dots, S\}$. \qed 
 
\subsection{Proof of Proposition~\ref{prop:ao-single-lac}. Boundary size}

For each $i \in \{1, \dots, S\}$,
we find 
\[
\phi_i \in \Phi_{n}^{4\alpha}(Q_i) 
\]
where $Q_i = Q_{P_i}$ and $P_i$ is a top multitile of $\sigma(i)$ such that 
\[
c_i =  \no{\rho_{I_{\sigma(i)}}^{-\alpha} [ \phi_i* f ]}_2, \quad M\sqrt{\abs{ I_{\sigma(i)} }} \leq c_i \leq D M\sqrt{\abs{ I_{\sigma(i)} }}.
\]
Let
\[
g_i = \frac{\rho_{I_T}^{-\alpha} [ \phi_i * f ] }{\no{\rho_{I_T}^{-\alpha} [ \phi_i* f ]}_2}.
\]
Now
\begin{multline*}
\sum_{i = 1}^S M^2 \abs{ I_{\sigma(i)} } \lesssim 
\sum_{i = 1}^S c_i \abs{ I_{\sigma(i)} }^{1/2} \la{f, \bar \phi_i* (\rho_{I_{\sigma(i)}}^{-\alpha} g_i ) } \\
\le \no{f}_{2} \nos{\sum_{i = 1}^S c_i \abs{ I_{\sigma(i)} }^{1/2}  [ \phi_i*  (\rho_{I_{\sigma(i)}}^{-\alpha} g_i ) ] }_2.
\end{multline*}
Fix $\kappa \in \{0, \ldots,99\}$.
Write $\mathcal{L} = \{i \in \{1, \dots, S\} \colon \log_2 \abs{ I_i }^{1/d} \in \kappa + 100\Z  \}$.
Expanding the square and using the symmetry,
we obtain
\begin{multline*}
\label{eq:2tree-sel-2}
\nos{ \sum_{i \in \mathcal{L} } \abs{ I_{\sigma(i)} }^{1/2} c_i [ \phi_i* (\rho_{I_T}^{-\alpha} g_i ) ] }_2^{2} \\
\lesssim D^2 M^{2} \sum_{i \in \mathcal{L} } \abs{ I_{\sigma(i)} }^{1/2} \sum_{\substack{l \in \mathcal{L}\\ \abs{I_l} \le \abs{I_i} }} 
\abs{I_{\sigma(l)}}^{1/2} \la{g_{l} \rho_{I_{\sigma(l)}}^{-\alpha}, \phi_{l}  * [ \phi_i *(\rho_{I_{\sigma(i)}}^{-\alpha} g_i) ]}.
\end{multline*}

Let 
\[
\mathcal{A}_i  = \{ l \in \mathcal{L}:\ l \in \{ i, \dots, S \} , \ \supp \widehat{\phi}_{l} \cap \supp \widehat{\phi}_{i} \ne \varnothing  \}.
\]
By Proposition~\ref{prop:commute-phi-rho}, Cauchy--Schwarz inequality and 
the estimates for the Hardy--Littlewood maximal function 
as above,
it suffices to prove a bound by constant of 
\[
\nos{\frac{1}{ \abs{ I_{\sigma(i)} }^{1/2}} \sum_{l \in \mathcal{A}_i} 
\abs{ I_{\sigma(l)} }^{1/2} g_{l} \rho_{I_{\sigma(l)}}^{-\alpha} \rho_{I_{\sigma(i)}}^{-\alpha} }_{2} .
\]
Indeed, by triangle inequality we can then sum over $\kappa \in \{0,\ldots,99\}$
to conclude the proof.
By Proposition~\ref{prop:overlap},
it hence remains to show that $\{I_{\sigma(l)} \colon l \in \mathcal{A}_i \}$ is a Carleson family.

Given $l,j \in \mathcal{A}_i$ with $l > j \geq i$, hence $\abs{ I_{\sigma(l)} } \leq \abs{ I_{\sigma(j)} }$, we have
\[
(Q_j)_n \cap  (Q_i)_n  \ne \varnothing , \quad (Q_{l})_n \cap  (Q_i)_n  \ne \varnothing .
\]
Therefore, we have
\[
I_{\sigma(l)} \cap I_{\sigma(j)} \neq \varnothing
\]
only if $\abs{ I_{\sigma(l)} } = \abs{ I_{\sigma(j)} }$ as otherwise Proposition~\ref{prop:almorth-projections}
would imply
\[
2^{k_2+1} Q_{l} % \supset 2^{k_2 - k_0 - 100} 2^{k_0+5}Q_l 
\supset 2^{k_2+1}Q_j \ni \xi_{\sigma(j)},
\]
which in turn would contradict $P_l \in \mathcal{V}_{\sigma(l)}$.
Therefore $I_{\sigma(j)}$ and $I_{\sigma(l)}$ are pairwise disjoint unless $I_{\sigma(j)} = I_{\sigma(l)}$. However, as above, for every fixed $ I_{\sigma(j)}$ there are only up to $C(d,k_0)$ elements $l \in \mathcal{A}_i$ such that $I_{\sigma(j)} = I_{\sigma(l)}$.
Hence $\{I_{\sigma(l)} \colon l \in \mathcal{A}_i \}$ is a Carleson family
and the proof is complete.  \qed
 
\subsection{Proof of Proposition~\ref{prop:ao-sum-lac}. Sum size}
Consider an index $i \in \{1, \dots, S\}$.
For $\theta \ge 1$,
we define 
\[
C_i(\theta) = \{\xi \in \R^{d} \colon \abs{ \xi-\xi_{\sigma(i)} } \le  \theta(\xi - \xi_{\sigma(i)}) \cdot e \}.
\]
Denote  
\[
a_i(j,\theta) = C_i(\theta) \cap  (Q_n(\xi_{\sigma(i)},2^{-j+1}) \setminus Q_n(\xi_{\sigma(i)},2^{-j})).
\]
We let 
\[
\mathcal{B}_j^{i} = \{ P \in \mathcal{B}_{\sigma(i)} \colon (Q_{P})_n \cap a_{i}(j,2) \ne \varnothing  \}.
\]
We note that if $P \in \mathcal{B}_j^{i}$,
then by Proposition~\ref{prop:lacunary-trivial}
\[
(Q_P)_n \subset \bigcup_{k = j-50k_2}^{j+50k_2} a_i(k,10). 
\]
For each $P \in \mathcal{B}_j^{i}$, 
we find $\phi_{P} \in M_n(\xi_{\sigma(i)},(Q_{P})_n)$ with $\widehat{\phi}_{P} \subset (Q_{P})_n$
such that for 
\[
c_{P} = \no{\rho_{I_{P}}^{-\alpha} [ \phi_{P}*f ]}_2
\]
it holds 
\[
\sum_{j \in \Z} \sum_{P \in \mathcal{B}_j^{i}} c_{P}^{2} \gtrsim M^{2}\abs{ I_{\sigma(i)} }, \quad c_P \leq D M \sqrt{\abs{ I_{P} }}.
\]
Let
\[
g_P = \frac{\rho_{I_P}^{-\alpha} [ \phi_P  * f ]}{\no{\rho_{I_P}^{-\alpha} [ \phi_P * f ] }_2}
\]
if $P \in \mathcal{B}_j^{i}$ for some $i$ and $j$ and $g_P = 0$ otherwise.

Now 
\begin{multline*}
\sum_{i = 1}^S M^{2} \abs{ I_{\sigma(i)} } 
\lesssim \sum_{i = 1}^S \sum_{j \in \Z} \sum_{P\in\mathcal{B}_j^{i}} c_{P} \la{f,\bar \phi_P* (\rho_{I_{P}}^{-\alpha}g_P)} \\
\le \no{f}_{2} \nos{ \sum_{i = 1}^S \sum_{j \in \Z} \sum_{P\in\mathcal{B}_j^{i}} c_{P} [ \phi_P*  (\rho_{I_{P}}^{-\alpha} g_P ) ] }_2.
\end{multline*}
By triangle inequality,
we may restrict the sum over $j \in \Z$ to a sum over $j \in \kappa + 1000k_2 \Z$ 
and integer $\kappa$.
For fixed $\kappa$ and every $i \in \{1,\dots, S\}$ we define
\[
\mathcal{E}^i_\kappa = \bigcup_{j \in \kappa + 1000k_2 \Z} \mathcal{B}_j^{i}.
\]
Squaring the second factor and using symmetry,
we compute
\begin{multline}
\label{eq:3tree-sel-2}
% \nos{ \sum_{i = 1}^S \sum_{j \in \kappa + 100k_2 \Z} \sum_{P\in\mathcal{B}_j^{i}} c_{P} [ \phi_P*  (\rho_{I_{P}}^{-\alpha} g_P ) ] }_2^{2} \lesssim \left( \sup_{P \in \bigcup_{i=1}^S \bigcup_{j \in \kappa + 100 k_2 \Z} \mathcal{B}_j^{i} } \frac{c_P^{2}}{\abs{ I_P }} \right) \\
% \times \sum_{s = 1}^S \sum_{P\in\mathcal{P}_{\sigma(s)}} \abs{ I_P }^{1/2}
% \sum_{l = 1}^S \sum_{\substack{P'\in \mathcal{P}_{\sigma(l)}\\ \abs{ I_{P'} } \le \abs{ I_P }}} \abs{ I_{P'}}^{1/2}
% \la{ \rho_{I_{P'}}^{-\alpha}g_{P'}, \phi_{P'}* [ \phi_P*(\rho_{I_{P}}^{-\alpha} g_P) ] }. 
\nos{ \sum_{i = 1}^S \sum_{P\in \mathcal{E}^{i}_\kappa} c_{P} [ \phi_P*  (\rho_{I_{P}}^{-\alpha} g_P ) ] }_2^{2} \lesssim \left( \sup_{P \in \bigcup_{i=1}^S \mathcal{E}^{i}_\kappa}  \frac{c_P^{2}}{\abs{ I_P }} \right) \\
\times \sum_{s = 1}^S \sum_{P\in \mathcal{E}^{s}_\kappa} \abs{ I_P }^{1/2}
\sum_{l = 1}^S \sum_{\substack{P'\in \mathcal{E}^{l}_\kappa\\ \abs{ I_{P'} } \le \abs{ I_P }}} \abs{ I_{P'}}^{1/2}
\la{ \rho_{I_{P'}}^{-\alpha}g_{P'}, \phi_{P'}* [ \phi_P*(\rho_{I_{P}}^{-\alpha} g_P) ] }. 
\end{multline} 
Fix $s$ and $P \in \mathcal{P}_{\sigma(s)}$ and let
\[
\mathcal{A}_{P} =  \left\{ P' \in \bigcup_{j \le 0} \mathcal{B}_{j_{\sigma(s)} + 1000j}^{s} \colon  \supp \widehat{\phi}_{P} \cap \supp \widehat{\phi}_{P'} \ne \varnothing  \right\} .
\]
By Proposition~\ref{prop:lacunary-trivial},
we may apply Proposition~\ref{prop:commute-phi-rho}
to bound 
\[
\phi_{P'}* [ \phi_P*(\rho_{I_{P}}^{-\alpha} g_P) ] \lesssim \rho_{I_P}^{-\alpha} \mathcal{M}_{\HL} g_P.
\]
By Cauchy--Schwarz inequality and the Hardy--Littlewood maximal function theorem as above,
we hence obtain 
\begin{multline}
\label{e:sumpart-single-tile}
\sum_{ P'\in \mathcal{A}_{P} } \abs{ I_{P'} }^{1/2}
\la{ \rho_{I_{P'}}^{-\alpha}g_{P'}, \phi_{P'}* [ \phi_P*(\rho_{I_{P}}^{-\alpha} g_P) ] }\\
\lesssim 
\sum_{ P'\in \mathcal{A}_{P} } \abs{ I_{P'} }^{1/2}
\la{ \rho_{I_{P}}^{-\alpha} \rho_{I_{P'}}^{-\alpha}g_{P'}, \mathcal{M}_{\HL} g_P } \\
\lesssim
 \nos{ \sum_{ P'\in \mathcal{A}_{P} } \abs{ I_{P'} }^{1/2}
\rho_{I_{P}}^{-\alpha} \rho_{I_{P'}}^{-\alpha}g_{P'} }_{2}
\lesssim \abs{ I_{P} }^{1/2} \rho_{I_P} ( \R^{d} \setminus I_{\sigma(s)})
\end{multline}
where the last inequality followed by Proposition~\ref{prop:sumpart-disjointness} below, the fact that for every fixed $ I_{P}$ there are only up to $C(d,k_0)$ elements $P' \in \mathcal{A}_P$ such that $I_{P} = I_{P'}$, and 
Proposition~\ref{prop:overlap}.

\begin{proposition}
\label{prop:sumpart-disjointness}
Assume that $L \in \sigma(l)$, 
$H \in \sigma(h)$
and $L,H \in \mathcal{A}_P$.
Assume additionally that $\abs{ I_H } < \abs{ I_L } < \abs{ I_P }$.
Then 
\[
I_L \cap I_H = \varnothing, \quad (I_L \cup I_H) \cap I_P = \varnothing.
\] 
\end{proposition}
\begin{proof}
Because $L,H \in \mathcal{A}_P$,
\[
(Q_L)_n  \cap    (Q_P)_n   \ne \varnothing, \quad  (Q_H)_n \cap (Q_P)_n \ne \varnothing.
\]
As $2^{2000k_2} \abs{ I_H } \leq 2^{1000k_2} \abs{ I_L } \leq \abs{ I_P }$,
this implies by Proposition~\ref{prop:almorth-projections} 
\begin{equation}
\label{e:sumpart-disjointness}
% 2^{k_2+1} Q_s \subset 2^{k_2+1} Q_l \subset 2^{k_2+1} Q_h.
\xi_{\sigma(s)} \in 2^{k_2+1} Q_{P} \subset 2^{k_2+1} Q_{L} \subset 2^{k_2+1} Q_{H}.
\end{equation}
Further, 
\[
a_l(j,10) \cap a_h(j',10) \ne \varnothing
\]
with $j' \le j - 10k_2$ only if $\xi_{\sigma(l)} \cdot e > \xi_{\sigma(h)} \cdot e $.
Hence $s < l < h$ and the claim follows by \eqref{e:sumpart-disjointness} as otherwise it would contradict $L \in \mathcal{V}_{\sigma(l)}$ and $H \in \mathcal{V}_{\sigma(h)}$.
\end{proof}

Applying the estimate \eqref{e:sumpart-single-tile} to the second factor on the right hand side of \eqref{eq:3tree-sel-2},
we obtain  
\begin{multline*}
\sum_{s = 1}^S \sum_{P\in % \mathcal{P}_{\sigma(s)}
\mathcal{E}^{s}_\kappa }\abs{ I_P }^{1/2} \nos{ 
\sum_{P' \in \mathcal{A}_P  } \abs{ I_{P'} }^{1/2}
\rho_{I_{P}}^{-\alpha} \rho_{I_{P'}}^{-\alpha}g_{P'} }_{2} \\
\lesssim \sum_{s = 1}^S \sum_{P\in\mathcal{P}_{\sigma(s)}} \abs{ I_P } \rho_{I_{P}}^{-\alpha}(\R^{d} \setminus I_{\sigma(s)}) 
\lesssim \sum_{s = 1}^S \abs{ I_{\sigma(s)} } \sum_{k=0}^{\infty} k2^{-k}
\lesssim  \sum_{s = 1}^S \abs{ I_{\sigma(s)} } . 
\end{multline*}
This concludes the proof. \qed

\subsection{Proof of Proposition~\ref{prop:induction-step}. Recursion} 
To streamline the language,
we introduce the following definition.
\begin{definition}[Admissible tree]
Let $\mathcal{V}$ be a finite subset of multitiles.
A tree $T = T(\xi,I_0,\mathcal{V})$ is said to be $n$-admissible with respect to 
boundary size if 
\[
\Sigma_{n,2,f_n}^{\BDR}(T) \le  \Sigma_{n,2,f_n}^{\BDR,\TOP}(T) .
\]  
It is said to be 
be $n$-admissible with respect to 
core size if 
\[
\Sigma_{n,2,f_n}^{\COR}(T) \le \Sigma_{n,2,f_n}^{\COR,\TOP}(T) .
\]  
\end{definition}

\begin{proposition}
%\label{prop:convex}
Let $N,N' > 0$. The family of multitiles 
% \[
% \mathcal{V} = \{ P \colon Q_P \in \mathcal{W}_N , \ I_P = 2^k([0,1)^d + l), \abs{l} \leq N' \} 
% \]
\[
\mathcal{V} = \{ P \colon Q_P \in \mathcal{W}_N , \ I_P \subset [ - N' 2^N, N' 2^N ]^{3 \times d} \}
\]
with $\mathcal{W}$ in Proposition~\ref{whitneyprop} is a convex collection.
If $\sigma$ is a selection on a convex collection,
then $\mathcal{V}_{\sigma(i)}$ is a convex collection for all $i \in \{1, \dots, S\}$.
\end{proposition}
\begin{proof}
Clear.
\end{proof}

Now we can proceed to the actual proof.
We define the selection on $\mathcal{V}$ as follows. 
For notational purposes,
we set $I_{\sigma(0)} = \R^{d}$, 
$\mathcal{P}_{\sigma(0)} = \varnothing$ and
$\mathcal{V}_{\sigma(0)} = \mathcal{V}$.
Finally, without loss of generality and only for notational convenience, 
assume $\no{f_n}_{2} = 1$ for all $n$.
 
Suppose $\sigma(i-1)$ has been defined.
A tree $T$ is called $X$-tree if it holds $ X(T)\ge 2^{(M-10d)/2}$ for some size $X$.
We first define the selection by choosing repeatedly $\Sigma_{n,2,f_n}^{\COR,\TOP}$-trees with $n=1$
such that trees with larger top cubes are chosen first,
only admissible trees are chosen
and only trees $T$ with $\Theta(T) \ne 0$ are chosen.
We denote by $s_n$ the number of steps at which we reach the last tree chosen.
This number is finite as there are only finitely many multitiles in the original collection $\mathcal{V}$.

We replace $\mathcal{V}$ with $\mathcal{V}_{\sigma(s_n)}\setminus\mathcal{P}_{\sigma(s_n)}$. We repeat the same process with $X$ replaced by $\Sigma_{n,2,f_n}^{\COR,\TOP}$ first with $n = 2$ and then with $n=3$.
This way,
we create three selections.
The first of them is $\sigma_1$ on $\mathcal{V}$.
The second one is $\sigma_2$ on $\mathcal{V}_{\sigma(s_1)}\setminus\mathcal{P}_{\sigma(s_1)}$
and the third one is $\sigma_3$ on $\mathcal{V}_{\sigma(s_2)}\setminus\mathcal{P}_{\sigma(s_2)}$.
By Proposition~\ref{prop:ao-core},
each of these selections satisfy \eqref{e:induction-step-overlap}. 
Set
\[
\mathcal{V}_1 = \mathcal{V} \setminus \bigcup_{n=1}^{3} \bigcup_{i = s_{n-1}+1 }^{s_{n}} \mathcal{P}_{ \sigma(i)}.
\]
A tree $T$ on $\mathcal{V}_1$
is either inadmissible or satisfies 
\[
 \Theta(T)\Sigma_{n,2,f_n}^{\COR,\TOP}(T) \le 2^{(M-10d)/2}.
\]
For admissible trees the latter condition is the desired size bound.
On the other hand,
if an inadmissible tree violates the size bound,
then by convexity of the reference family it contains an admissible subtree violating the size bound.
But this possibility was just ruled out.
This concludes the treatise with respect to the core size.

We repeat the same selection process with $\Sigma_{n,2,f_n}^{\BDR,\TOP}$ instead of $\Theta \Sigma_{n,2,f_n}^{\COR,\TOP}$
and this gives us three more selections $\sigma_{4}$, $\sigma_{5}$ and $\sigma_{6}$ and a family $\mathcal{V}_2$
such that the trees in selections satisfy \eqref{e:induction-step-overlap} by Proposition~\ref{prop:ao-single-lac}
and by the argument as in the case of the core size also the size bound is valid.

It remains to treat the sum size.
Let $\{e_{\delta} \colon 1 \le \delta \le 2d  \}$ be the unit vectors orthogonal to the $(d-1)$-dimensional hyperplanes.
Let $\mu^{\delta,n} \in M_n(C_{e_{\delta}})$ with 
\[
\sum_{\delta=1}^{2d} \widehat{\mu}^{\delta,n}(\xi) = 1, \quad \xi \ne 0.
\] 
For a tree $T$ on $\mathcal{V}_2$ or any of its subfamilies,
we set $\widehat{\mu}^{\delta}_{T,n}(\xi) = \widehat{\mu}^{\delta,n}(\xi-(\xi_T)_n)$.
We run the selection choosing trees $T$ such that
\[
\Sigma_{n,2,\mu^{\delta,n} * f_n}^{\SUM}(T) \ge 2^{(M-10d)/2}
\]
so that ones with maximal $e_{\delta}\cdot (\xi_{T})_n$ are chosen first.
Again,
we repeat this process for each $n$ and each $\delta$.
Each of the $6d$-selections satisfy the hypotheses of Proposition~\ref{prop:ao-sum-lac},
and the trees not chosen satisfy \eqref{eq:sele-2}. 
Collecting the trees in all of the selections constructed so far,
we obtain the family $\mathcal{T}_{M}$ and the proof is complete. \qed

\subsection{Conclusion of the proof of Proposition~\ref{almostorthogonal}}
Because the family of multitiles $\mathcal{V}$ in the hypothesis of the proposition is finite,
there exists $M$ such that the hypothesis of Proposition~\ref{prop:induction-step} holds.
The claim except for \eqref{e:size-by-linfty} follows by induction.

To prove \eqref{e:size-by-linfty},
we first note that for any tree $T$ and any $p \in [1,\infty]$
\[
\Sigma_{n,p,f_n}^{\COR}(T) + \Sigma_{n,p,f_n}^{\BDR}(T) \le C(d) \no{f_n}_{\infty}
\]
is obvious. 
It remains to bound the sum size.

Let $\eta$ be a smooth function with $\eta \gtrsim 1_{I_T}$ and $\supp \widehat{\eta} \subset Q_{n_{*}}(0,2^{-j_T})$.
Let $\{\varphi_P \in \Phi_{n}^{4\alpha}(Q_P) \colon P \in \mathcal{B}_T \}$ be functions to almost achieve the supremum in the definition of the sum size. 
First, we note 
\begin{multline*}
\sum_{P \in \mathcal{B}_T} \no{1_{I_P} [ \varphi_P *( 1_{\Rd \setminus 3I_T} f_n ) ] }_{2}^{2} \\
\lesssim 
\no{f_n}_{\infty} \sum_{P \in \mathcal{B}_T} \abs{ I_{I_P} } \rho_{I_P}^{-\alpha}(\Rd \setminus 3I_T)
\lesssim \abs{ I_T } \no{f_n}_{\infty} .
\end{multline*}
Second, we note 
% \begin{equation}
% \label{e:sum-inproof-1}
% \sum_{P \in % \mathcal{P}_T
% \textcolor{blue}{\mathcal{B}_T}} \no{1_P [ \varphi_P *( 1_{3I_T} f_n ) ] }_{2}^{2}
% \lesssim \sum_{j}  \no{1_T [ \varphi_j *( 1_{3I_T} f_n ) ] }_{2}^{2}
% \end{equation}
% for a sequence $\varphi_j \in \Psi_{n,j}^{4\alpha}(\xi_T)$.
\begin{equation}
\label{e:sum-inproof-1}
\sum_{P \in \mathcal{B}_T} \no{1_{I_P} [ \varphi_P *( 1_{3I_T} f_n ) ] }_{2}^{2}
\lesssim \sum_{k = k_1}^{k_2} \sum_{j \leq j_T}  \no{1_T [ \varphi_{j,k} *( 1_{3I_T} f_n ) ] }_{2}^{2}
\end{equation}
for a family of sequences $\{ \{ \varphi_{j,k} \in \Psi_{n,j - k}^{4\alpha}(\xi_T) \colon j \leq j_T \} \colon k \in \{k_1, \dots, k_2 \} \} $. Without loss of generality we fix $k$ and we drop it from notation.
For terms with $j_T - j \le 100$,
we estimate 
\[
\no{1_T [ \varphi_j *( 1_{3I_T} f_n ) ]}_{2}^{2} \lesssim \abs{ I_T } \no{f_n}_{\infty}
\]
as in the cases of core and boundary sizes.
For terms with $j_T - j > 100$
we note that % by Proposition~\ref{prop:lacunary-trivial}
\[
[ \R^d \setminus Q_n(\xi_{T},2^{-j + k -100}) ] \supset ( \supp \widehat{\eta} + \supp \widehat{\varphi}_j ) \supset \supp ( \widehat{\eta} * \widehat{\varphi}_j ). 
\]

Consequently it holds
$\supp( \widehat{\eta} * \widehat{\varphi}_j ) \cap \supp ( \widehat{\eta} * \widehat{\varphi}_{j'} ) = \varnothing$
whenever $\abs{ j-j' } \ge 100$, 
and $\sum_{j < j_T - 100} \widehat{\eta} * \widehat{\varphi}_j$
is a Mikhlin multiplier with bounds only depending on the dimension.

We bound each inner sum in \eqref{e:sum-inproof-1}
by 
\begin{multline*}
 \sum_{j < j_T - 100}  \no{\eta [ \varphi_j *( 1_{3I_T} f_n ) ] }_{2}^{2}
\lesssim \sum_{l=1}^{100} \sum_{\substack{j \in l + 100 \Z\\ j < j_T - 100 }}  \nos{ \eta [ \varphi_j *( 1_{3I_T} f_n ) ] }_{2}^{2} \\
\lesssim \sum_{l=1}^{100}  \nos{ \eta \sum_{\substack{j \in l + 100 \Z\\ j < j_T - 100 }}  [ \varphi_j *( 1_{3I_T} f_n ) ] }_{2}^{2}
\lesssim \no{1_{3I_T}f_n}_{2}^{2}
\end{multline*}
where the last step followed by the Mikhlin multiplier theorem.
The claim then follows. \qed

\section{Proof of Proposition~\ref{whitneyprop}. Tensorized model form} 
\label{s:whitneyprop}

By dominated convergence, there exists $N' = N' (d, \alpha, \varepsilon, k_0, f_n, N) > 0$ such that for the finite subset of multitiles % $\mathcal{V} = \{ P \colon Q_P \in \mathcal{W}_N, I_P = 2^k ([0,1)^d + l ), \abs{l} \leq N' \}$ 
$ \mathcal{V} = \{ P \colon Q_P \in \mathcal{W}_N , \ I_P \subset [ - N' 2^N, N' 2^N ]^{3 \times d} \} $ we have
\begin{multline*}
  \abs{ \sum_{ Q\in \mathcal{W}_N } 
  \int_{\R^d} \prod_{n=1}^3 [ \phi_{Q,n} * {f}_n (x)] \, \diff x } 
  \\
\leq 2 
 \abs{ \sum_{P \in \mathcal{V}} \int_{\R^{d}} 1_{I_P}(x) \prod_{n=1}^{3} [ \phi_{P,n}* f_n (x) ] \, \diff x}  .
\end{multline*}

Consider the families of trees $\mathcal{T}_M$ as in Proposition~\ref{almostorthogonal}.
By triangle inequality, % we write 
we get an upper bound by
\begin{equation*}
%\begin{multline*}
\label{E:main2}
%\left|\sum_{(\xi,j) \in \mathcal{V}} 
%  \int_{\R^d} \left[\prod_{n=1}^3 {f}_n*  \phi_{\xi,j,n}(x) \right]\, \diff x\right|
% \abs{ \sum_{\substack{Q(\xi, 2^{-j}) \in {\mathcal{W}}\\ \abs{\xi},\abs{j}\le N}} 
%
%\abs{ \sum_{\textcolor{blue}{Q\in \mathcal{W}_N}} 
%  \int_{\R^d} \prod_{n=1}^3 [ \phi_{Q,n} * {f}_n (x)] \, \diff x } 
%  \\
%\leq 2 
 \sum_{M \in \N \cup \{-\infty\} } \sum_{T \in \mathcal{T}_M} \abs{ \sum_{P \in \mathcal{P}_T} \int_{\R^{d}} 1_{I_P}(x) \prod_{n=1}^{3} [ \phi_{P,n}* f_n (x) ] \, \diff x}  .
%\end{multline*}
\end{equation*}
By Proposition~\ref{thm:tree-estimate},
we get an upper bound by 
\begin{equation}
\label{e:whitney-proof-1}
\sum_{M \in \N \cup \{-\infty\} } \sum_{T \in \mathcal{T}_M} \abs{ I_T } \left( \Sigma_{n_{*}, % q_{n_{*}}
\infty,f_{n_{*}}}^{\BDR}(T)   \prod_{n \ne n_{*}} \Sigma_{n,f_n}^{\SUM}(T)  +  \prod_{n = 1}^{3} \Sigma_{n,q_n,f_n}^{\COR}(T) \right),
\end{equation}
where the second summand in brackets appears if and only if $\mathcal{P}_T \setminus \mathcal{B}_T \neq \varnothing$, i.e. if there exists $P \in \mathcal{P}_T$ with $2^{k_1+1}Q_P \ni \xi_{T}$.

By log-convexity and \eqref{e:size-by-linfty} from Proposition~\ref{almostorthogonal},
we have  
% \begin{align*}
% \Sigma_{n,q_{n},f_n}^{\BDR}(T) &\lesssim \Sigma_{n,2,f_n}^{\BDR}(T)^{2/q_{n}}, \\ % &\le \no{f_n}_{\infty}^{1-2/q_{n}} \Sigma_{n,2,f_n}^{\BDR}(T)^{2/q_{n}}, \\
% \Sigma_{n,q_{n},f_n}^{\COR}(T) &\lesssim \Sigma_{n,2,f_n}^{\COR}(T)^{2/q_{n}}, \\ % &\le \no{f_n}_{\infty}^{1-2/q_{n}} \Sigma_{n,2,f_n}^{\COR}(T)^{2/q_{n}}, \\
% \Sigma_{n,f_n}^{\SUM}(T) &\lesssim \Sigma_{n,f_n}^{\SUM}(T)^{2/q_{n}}. % &\le \no{f_n}_{\infty}^{1-2/q_{n}} \Sigma_{n,f_n}^{\SUM}(T)^{2/q_{n}}.
% \end{align*}
\begin{equation*}
\Sigma_{n,q_{n},f_n}^{\COR}(T) \lesssim \Sigma_{n,2,f_n}^{\COR}(T)^{2/q_{n}}, \quad \Sigma_{n,f_n}^{\SUM}(T) \lesssim \Sigma_{n,f_n}^{\SUM}(T)^{2/q_{n}}.
\end{equation*}
By the local Bernstein's inequality (see e.g. Proposition~1.2 in \cite{fraccaroli2022phase})
\[
\Sigma_{n,\infty,f_n}^{\BDR}(T) \lesssim 2^{dv_n/2} \Sigma_{n,2,f_n}^{\BDR}(T), \quad \Sigma_{n,q_{n},f_n}^{\COR}(T) \lesssim 2^{dv_n (1/2 - 1/q_n )} \Sigma_{n,2,f_n}^{\COR}(T) .
\]
In addition,
we know by Proposition~\ref{almostorthogonal} that for all $n \in \{1,2,3\}$ and $T \in \mathcal{T}_M$
\begin{align*}
% \Sigma_{n,2,f_n}^{\COR}(T) + \Sigma_{n,2,f_n}^{\BDR}(T)  \lesssim   \min \{1,2^{M/2} \no{f_n}_2 \}.
\Sigma_{n,2,f_n}^{\BDR}(T) & \lesssim   \min \{1,2^{M/2} \no{f_n}_2 \}, \\
\Sigma_{n,f_n}^{\SUM}(T) & \leq 2^{M/2} \no{f_n}_2, \\
\Sigma_{n,2,f_n}^{\COR}(T) & \lesssim 1.
\end{align*}
Moreover, if there exists $P \in \mathcal{P}_T$ with $2^{k_1+1}Q_P \ni \xi_{T}$ we also have
\[
\Sigma_{n,2,f_n}^{\COR}(T) \le 2^{M/2} \no{f_n}_2 .
\]
Recalling that there is $n_{*}$ such that $v_{n_{*}} = 0$,
we hence bound \eqref{e:whitney-proof-1} by 
\begin{multline*}
\sum_{M \in \N \cup \{-\infty\} } \sum_{T \in \mathcal{T}_M} \abs{I_T} \min \{1,2^{M/2} \no{f_{n_{*}}}_2 \} \prod_{n \ne n_{*}} 2^{M/q_{n}} \no{f_{n}}_2^{2/q_{n}} \\
\lesssim \sum_{M \in \N \cup \{-\infty\} }  \min \{2^{-M/q_{n_{*}}},2^{M(1/2-1/q_{n_{*}})} \no{f_{n_{*}}}_2 \} \prod_{n \ne n_{*}} \no{f_{n}}_2^{2/q_{n}} \\
\lesssim \prod_{n =1}^{3} \no{f_{n}}_2^{2/q_{n}}
\end{multline*}
where the first inequality used \eqref{e:almorth-overlap} from Proposition~\ref{almostorthogonal}.
This concludes the proof of Proposition~\ref{whitneyprop}. \qed

\section{Proof of  Theorem~\ref{main-theorem}}
\label{s:main-theorem}
 
The integral 
\begin{equation}
\label{eq:main-thm-int}
%\iiint_{  \xi_1+ \xi_2 + \xi_3 = 0} m(\xi_1,\xi_2,\xi_3) \prod_{n=1}^{3} \widehat{f}_n(\xi_n) \, \diff \xi_1 \diff \xi_2 \diff \xi_3 
\int_{\R^{3 \times d}} \delta_0(\xi_1+\xi_2+\xi_3) \widehat{f}_1(\xi_1)\widehat{f}_2(\xi_2)\widehat{f}_3(\xi_3) m(L^{-1} \xi) \, \diff \xi
\end{equation}
is absolutely convergent for Schwartz functions $f_1$,$f_2$ and $f_3$. 
Approximating $m$ with a symbol supported in a compact set not meeting 
$\{ (\tau, \tau, \tau ) \colon \tau \in \R^d \}$,
we conclude by the Lebesgue dominated convergence theorem,
boundedness of $m$
and absolute convergence of the integral 
that it suffices to assume $m$ is compactly supported.

By multilinear interpolation \cite{MR0942274},
it suffices to prove a bound for the integral \eqref{eq:main-thm-int} by 
\[
C \prod_{n=1}^{3} \no{f_n}_{q_n} 
\]
when $f_n = 1_{E_n}$ for a measurable sets $E_n$ of finite measure 
and $C$ is a constant independent of all $E_n$ and $m$.
Because $m$ is assumed to be compactly supported,
the integral \eqref{eq:main-thm-int} is absolutely convergent even with $f_n = 1_{E_n}$.
By standard convolution approximation and dominated convergence theorem,
we see that it suffices to bound the integral \eqref{eq:main-thm-int} by 
\[
C \prod_{n=1}^{3} \abs{ E_n }^{1/q_n}
\]
whenever $f_n$ is a smooth function with 
\[
\no{f_n}_{\infty} \le 2, \quad \no{f_n}_2 \le 2 \abs{ E_n }^{1/2}.
\]
Indeed, the convolution mollification converges in all $L^{p}$ norms with $p$ finite,
in particular with $p \in \{2,q_n\}$.
We see that for each $n \in \{1,2,3\}$ the function $f_n$ satisfies the assumptions on Proposition~\ref{whitneyprop}.

Next we form a Whitney type decomposition of $\R^{3 \times d} \setminus \Gamma$.
For each $\xi \in \R^{3 \times d} \setminus \Gamma$,
set 
\[
r_\xi = \frac{3}{4} \inf \{r > 0  \colon Q(\xi,r) \cap \Gamma \ne \varnothing \},
\]
where $Q(\xi,r) \subset \R^{3 \times d}$ is the open rectangular box defined in Section~\ref{sec:out}.
Let 
\[
\mathcal{A} = \{ Q(\xi,r)  \colon r = % 2^{-3-k_0}
2^{-k_0}
r_\xi  \}.
\]
We let $\mathcal{W}$ be a maximal pairwise disjoint family of $Q \in \mathcal{A}$
such that $Q \cap [-2^{N},2^{N}]^{3 \times d} \neq \varnothing$, where $N\ge 100$ is an integer such that $\supp m \subset L^{-1} ( [-2^{N},2^{N}]^{3 \times d} )$.
It is clear that 
% \[
% \{2^{k}Q  \colon k \le k_0+3, \ Q \in \mathcal{W}\}
% \]
for each fixed $k \leq k_0$
\[
\{2^{k} Q  \colon Q \in \mathcal{W} \}
\]
has bounded overlap.
Note that 
\[
\supp m \subset L^{-1} \left( \bigcup_{ Q \in \mathcal{W} } 5 Q \right) .
\] 

For $Q \in \mathcal{W}$,
we denote $Q_n = \{\xi_n \in \R^d \colon \xi \in Q\}$.
Let $\{\eta_{Q} \colon Q \in \mathcal{W}\} $ form a partition of unity adapted to $\mathcal{W}$,
meaning that for each $Q \in \mathcal{W}$ the smooth function $\eta_Q \ge 0$ is supported in $6Q$
and satisfies the bounds 
\[
\abs{ \partial_1^{\gamma_1} \partial_2^{\gamma_2} \partial_3^{\gamma_3} \eta_Q (\xi) } \le C_{\gamma} \abs{Q_1}^{-\abs{\gamma_1}/d} \abs{Q_2}^{-\abs{\gamma_2}/d} \abs{Q_3}^{-\abs{\gamma_3}/d}
\]
for constants $C_{\gamma}$ only depending on $\gamma = (\gamma_1,\gamma_2,\gamma_3) \in \N^{3\times d}$.

Let $\chi_{Q,n}$ be a smooth function with 
\[
1_{7Q_n} \le \chi_{Q,n} \le 1_{8Q_n}, \quad   \abs{ \partial^{\gamma} \chi_{Q,n} ( \tau) } \le C_{\gamma} \abs{Q_n}^{-\abs{\gamma}/d}
\] 
for all $\gamma \in \N^{ d}$ and $\abs{ \gamma }\le 100d$.
Let $\chi_Q(\xi) = \chi_{Q,1} (\xi_1)\chi_{Q,2} (\xi_2)\chi_{Q,3}(\xi_3)$ for $\xi \in \R^{3\times d}$.
% Let $A_Q$ be a linear mapping sending $7Q$ to $[0,2\pi)^{3 \times d}$.
% Such a matrix is of block form such that $A_Q = A_{Q,1} \oplus A_{Q,2} \oplus A_{Q,3}$.

Let $A_Q$ be a linear mapping sending $7Q - \mathfrak{P} (\xi_Q)$ into $[0,2\pi)^{3 \times d}$ and such that $A_Q ( [ 8Q - \mathfrak{P} (\xi_Q) ] ) \setminus (-2\pi,2\pi)^{3 \times d} \neq \varnothing$, where $\xi_Q$ is such that $Q = Q(\xi_Q,r)$ and $\mathfrak{P}$ is the orthogonal projection of $\R^{3 \times d}$ onto $\Gamma$.
Such a matrix is of block form $A_Q = A_{Q,1} \oplus A_{Q,2} \oplus A_{Q,3}$.
We expand as a Fourier series 
\[
m_Q( L^{-1} \xi) \coloneq \eta_{Q}(\xi)m( L^{-1} \xi ) = \chi_Q(\xi) \sum_{k \in \Z^{3\times d} } a_{Q,k} \prod_{n=1}^3 e^{2 \pi i k_n \cdot A_n \xi_n }
\]
so that $m = \sum_{Q \in \mathcal{W}} m_Q$.

Denote $m_{Q,k,n} ( L^{-1}_n \xi_n) = \chi_{Q,n} (\xi_n) e^{2\pi i k_n \cdot A_n \xi_n}$ and $a_k = \sup_{Q \in \mathcal{W}} \abs{a_{Q,k}}$. For the function $ \phi_{Q,k,n}$ defined by 
\[ 
\widehat{\phi}_{Q,k,n} (\tau) = (1 + \abs{k_n})^{- 4 \alpha} m_{Q,k,n} ( L^{-1}_{n} \tau)
\]
we have $ c \phi_{Q,k,n} \in \Phi_{n}^{4\alpha}(Q)$ up to a bounded multiplicative constant $c$ independent of $Q$, $k$, and $n$.

Now we can write 
\begin{multline*}
% \iiint_{  \xi_1+ \xi_2 + \xi_3 = 0} m(\xi_1,\xi_2,\xi_3) \prod_{n=1}^{3} \widehat{f}_n(\xi_n) \, \diff \xi_1 \diff \xi_2 \diff \xi_3  = \\
%= \sum_{k \in \Z^{3 \times d}} a_k   \sum_{Q \in \mathcal{W}} \iiint_{  \xi_1+ \xi_2 + \xi_3 = 0}  \prod_{n=1}^{3} m_{Q,k,n}\widehat{f}_n(\xi_n) \, \diff  \xi_1 \diff \xi_2 \diff \xi_3 = \\
%=  \sum_{k \in \Z^{3 \times d}} a_k   \sum_{Q \in \mathcal{W}} \iiint_{\R^{3 \times d}} \prod_{n=1}^{3} (\phi_{Q,k,n}*f)(x) \, \diff x_1 \diff x_2 \diff x_3.
\abs{ \int_{\R^{3 \times d}} \delta_0(\xi_1+\xi_2+\xi_3) \widehat{f}_1(\xi_1)\widehat{f}_2(\xi_2)\widehat{f}_3(\xi_3) m(L^{-1}\xi) \, \diff \xi } \\
\leq \sum_{k \in \Z^{3 \times d}} \abs{a_k} \abs{ \sum_{Q \in \mathcal{W}} \int_{\R^{3 \times d}} \delta_0(\xi_1+\xi_2+\xi_3)  \prod_{n=1}^{3} m_{Q,k,n} (L^{-1}_n \xi_n ) \widehat{f}_n(\xi_n) \, \diff  \xi } \\
\lesssim \sum_{k \in \Z^{3 \times d}} \abs{ a_k } (1 + \abs{k})^{12\alpha} \abs{ \sum_{Q \in \mathcal{W}} \int_{\R^{d}} \prod_{n=1}^{3} [ \phi_{Q,k,n}*f_n(x) ] \, \diff x } .
\end{multline*} 
By Proposition~\ref{whitneyprop},
this is bounded by 
\[
C \prod_{n=1}^{3}\abs{ E_n }^{1/q_n} \sum_{k \in \Z^{3 \times d}} \abs{ a_k } (1 + \abs{k})^{12\alpha} .
\]
By smoothness of the symbol $m$ and the upper bound on $\alpha$,
we know $\abs{ a_k } \le C \abs{ k }^{-12\alpha-3 d-1}$,
and hence the proof is complete. \qed

\bibliography{ref}
\bibliographystyle{abbrv}

\end{document}